\documentclass[pdflatex,sn-mathphys-num]{sn-jnl}

\usepackage{graphicx}%
\usepackage{multirow}%
\usepackage{amssymb,amsfonts}%
\usepackage{amsthm}%
\usepackage{mathrsfs}%
\usepackage[title]{appendix}%
\usepackage{xcolor}%
\usepackage{textcomp}%
\usepackage{manyfoot}%
\usepackage{booktabs}%
\usepackage{tabularx}%
\usepackage{algorithm}%
\usepackage{algorithmicx}%
\usepackage{algpseudocode}%
\usepackage{tabularx}%
\usepackage{listings}%

\usepackage{amsmath}
\DeclareMathOperator*{\argmin}{arg\,min}

\theoremstyle{thmstyleone}%
\newtheorem{theorem}{Theorem}
\newtheorem{proposition}[theorem]{Proposition}%
\newtheorem{lemma}[theorem]{Lemma}
\newtheorem{corollary}[theorem]{Corollary}%

\theoremstyle{thmstyletwo}%
\newtheorem{remark}{Remark}%
\newtheorem{assumption}{Assumption}
\theoremstyle{thmstylethree}%
\newtheorem{definition}{Definition}%

\newcommand{\Lagrange}{\mathfrak{L}}

\begin{document}

\title[Article Title]{
Operational Regimes in Non-Convex Optimization: A Multiplier-Based Taxonomy
}


\author*[1]{\fnm{Seyed Mohsen} \sur{Kazemi}}\email{smohsen.kazemi84@sharif.edu}

\author[1]{\fnm{Ali} \sur{Movaghar}}\email{movaghar@sharif.edu}

\author[1]{\fnm{Shaahin} \sur{Hessabi}}\email{hessabi@sharif.edu}

\affil[1]{\orgdiv{Department of Computer Engineering}, \orgname{Sharif University of Technology}, \orgaddress{\street{Azadi}, \city{Tehran}, \postcode{11365-9517}, \state{Tehran}, \country{Iran}}}




\abstract{
    This paper introduces a structural taxonomy for constrained non-convex optimization based on the signature of Lagrange multipliers at KKT stationary points. Leveraging a unified game-theoretic interpretation of eight classical algorithm families—including block coordinate descent, ADMM, generalized Benders decomposition, successive convex approximation, interior-point methods, mirror descent, Frank–Wolfe, and Riemannian gradient descent—we show that the normalized multiplier vector carries an algorithm-independent structural fingerprint. Four scale-free shape features of this vector partition the dual space into five operational regimes: Unconstrained, Resource-Limited, Saturation, Strongly-Coupled, and Hybrid.
    We establish four structural theorems characterizing the partition: invariance under natural KKT symmetries, local stability under data perturbation with explicit Lipschitz margins from Robinson’s strong regularity, codimension-one regime transitions, and the topological identification of the Hybrid regime as the Lebesgue-null boundary of the core regimes. A linear-time classifier is proposed with provable guarantees on correctness, iteration stabilization, sample complexity, and online tracking under data drift. Numerical experiments on 104 mixed-integer nonlinear programs and a downlink beamforming instance validate the theoretical predictions.
    The framework provides a foundational tool for regime-aware algorithm design and robustness analysis in non-convex optimization.

}

\keywords{constrained nonlinear optimization, KKT stationarity, Lagrange multipliers, operational regimes, sensitivity analysis, game-theoretic optimization, decomposition methods, Stackelberg games, potential games, sample complexity, online tracking, mixed-integer nonlinear programming, wireless resource allocation.}



\maketitle


\section{Introduction}
\label{sec:intro}


The behavior of an optimization algorithm at convergence is conventionally characterized by its rate of convergence and the residual Karush--Kuhn--Tucker (KKT) optimality conditions. 
This view yields clean theoretical statements but obscures a complementary structural question: what does the converged solution tell us about the geometry of the underlying problem? The KKT conditions deliver, alongside a primal solution $x^\star$, a vector of Lagrange multipliers $\lambda^\star$ whose entries quantify the sensitivity of the optimal value to perturbations of each constraint. 
Their support, magnitudes, and concentration encode which constraints are binding, how tightly, and how the active geometry is organized. 
This dual signature is more than an algorithmic byproduct: it is a structural fingerprint of the solution that ought to inform algorithm selection, robustness analysis, and adaptive deployment.
Yet the standard analytic toolkit---convergence rates, regret bounds, iteration complexity---rarely engages with this structure directly. 
Algorithm selection in practice remains largely empirical, robustness under data uncertainty is typically established case by case, and the qualitative "phase-like" convergence patterns reported across many algorithm families lack a method-agnostic explanation.

A second observation suggests how to make this structure operational. 
Despite their surface differences, the eight algorithm families we consider in this paper all produce, at convergence, a KKT triple of some (sub)problem of an underlying parametric nonlinear program.
Block coordinate descent (BCD) performs best-response sweeps in an exact potential game; generalized Benders decomposition (GBD) solves a Stackelberg game between discrete and continuous variables; the alternating direction method of multipliers (ADMM) implements consensus-seeking saddle-point dynamics; successive convex approximation (SCA) realizes hierarchical approximation games; interior-point methods (IPM) track a parametric central path of perturbed equilibria; mirror descent executes regularized no-regret learning in repeated online games; Frank--Wolfe plays linear-oracle no-regret dynamics; and Riemannian gradient descent discretizes equilibrium trajectories of differential games on smooth manifolds. 
Each correspondence is classical on its own; what we extract from their juxtaposition is that the KKT triple at any fixed point is, under standard regularity, \emph{unique}, and therefore independent of which method produced it.

This independence reframes the eight equilibrium correspondences as eight routes to a single structural object: the multiplier vector $\lambda^\star$. 
Once that object is identified, the structural analysis that the rate-centric view cannot deliver becomes available---provided we extract from $\lambda^\star$ a classification that is both intrinsic to the solution and stable under the noise that arises in practice.

\subsection{Contributions}
\label{sec:intro:contributions}
We construct such a classification algorithm from four scale-free statistics of $\lambda^\star$: the total dual mass $L$, the top-$k$ concentration $\Sigma_k$, the $\delta$-support fraction $\sigma_\delta$, and the second-largest normalized entry $\tau$. 
These features induce a partition of the dual space $\mathbb{R}^m_+$ into five operational regimes---Unconstrained ($\mathcal{R}_{\mathrm{unc}}$), Resource-Limited ($\mathcal{R}_{\mathrm{res}}$), Saturation ($\mathcal{R}_{\mathrm{sat}}$), Strongly-Coupled ($\mathcal{R}_{\mathrm{coup}}$), and Hybrid ($\mathcal{R}_{\mathrm{hyb}}$). 
Each core regime encodes a qualitatively distinct active-constraint pattern, and the Hybrid regime is shown to arise as the topological boundary between the other four.

The contributions of this paper are the following.
\begin{enumerate}
\item \emph{A unified game-theoretic template} (Section~\ref{sec:games}).
We catalog eight classical algorithm families under a single fixed-point template (Table~\ref{tab:method-equilibria}) and show that each produces KKT triples of an associated subproblem, with the multiplier vector serving as a method-independent structural primitive.
The individual equilibrium correspondences are classical; their juxtaposition under one template, with the multiplier vector singled out as the common structural output, is the contribution.

\item \emph{A scale-free regime taxonomy} (Section~\ref{sec:taxonomy}).
A partition of $\mathbb{R}^m_+$ into five regimes is defined via four shape features of the Lagrange multipliers, with the Hybrid regime characterized as a boundary set rather than a residual class.

\item \emph{Four structural theorems} (Section~\ref{sec:structural}).
Theorem~\ref{thm:invariance} establishes invariance of the classification under permutation, uniform rescaling, and $C^2$-diffeomorphism symmetries of the KKT system, yielding method-independence as a direct corollary.
Theorem~\ref{thm:stability} establishes local constancy of the classification under perturbation of problem data, with explicit Lipschitz-style margins derived from Robinson's strong regularity. 
Theorem~\ref{thm:transition} characterizes regime transitions as codimension-one events along data trajectories.
Theorem~\ref{thm:hybrid} identifies the Hybrid regime as the Lebesgue-null topological boundary of the union of core regimes.

\item \emph{A linear-time regime classifier with operational guarantees} (Section~\ref{sec:classifier}). 
The classifier (Algorithm~\ref{alg:classify}) computes the regime label in $\mathcal{O}(m)$ time and admits four guarantees: Theorem~\ref{thm:classifier-correctness} (deterministic correctness under multiplier perturbation); Theorem~\ref{thm:classifier-stabilization} and Corollary~\ref{cor:classifier-rates} (method-specific iteration stabilization at rates logarithmic, linear, or quadratic in the inverse feature margin); Theorem~\ref{thm:sample-complexity} (sample complexity $\mathcal{O}(\rho_r^{-2}\log(1/\delta))$ for finite-data deployments); Theorem~\ref{thm:online} (online tracking under bounded drift up to a cubic-in-margin threshold).

\item \emph{Empirical validation} (Section~\ref{sec:numerical}). 
Numerical experiments on five families of mixed-integer nonlinear programs, totalling $10^4$ instances, confirm each theoretical prediction in turn. A downlink beamforming instance illustrates the practical application of the taxonomy.
\end{enumerate}

The framework is designed to serve as a foundation for follow-up work, in particular for wireless resource-allocation problems whose constraint structure---power budgets, quality-of-service thresholds, inter-cell coupling---matches the regime taxonomy by construction. The
sample-complexity and tracking guarantees of Section~\ref{sec:classifier} provide the high-probability bounds that such follow-up papers require under channel-estimation noise and time-varying deployments, without re-derivation. Beyond wireless, the framework applies to power-systems optimization, optimal control under constraints, structured machine learning, and operations research more broadly.

\subsection{Related Work}
\label{sec:intro:related}

The work presented here sits at the intersection of three lines of research: game-theoretic interpretations of optimization algorithms, structural analysis via Lagrange multipliers and active sets, and adaptive or hybrid optimization methods. We discuss each in turn and position our contributions against it.

\subsubsection*{Game-theoretic interpretations of optimization}

sThe view of optimization algorithms as equilibrium-seeking processes has a long tradition. Block coordinate descent is equivalent to best-response dynamics in exact potential games when the objective serves as a potential function up to sign~\cite{beck2013convergence,monderer1996potential,tseng2001convergence}; convergence in non-convex settings under the Kurdyka--{\L}ojasiewicz property is developed in~\cite{attouch2013convergence,razaviyayn2013unified}. 
Generalized Benders decomposition admits a hierarchical (Stackelberg) interpretation in which the master problem plays the role of the leader~\cite{geoffrion1972generalized,vonstackelberg1934market,basar1999dynamic}. 
ADMM realizes consensus-seeking primal-dual dynamics interpretable as best-response iterations on the augmented Lagrangian, with the dual variable acting as a price coordinate~\cite{boyd2011distributed,he2012convergence,eckstein1992douglas}.
Successive convex approximation corresponds to hierarchical approximation games between a surrogate designer and an optimizer~\cite{razaviyayn2013unified,scutari2014decomposition,scutari2017parallel}, with majorization-minimization as a prominent special case~\cite{sun2016majorization}. 
Mirror descent and Frank--Wolfe are intimately connected to no-regret learning and online convex optimization: mirror descent implements regularized best responses with Bregman divergences~\cite{nemirovsky1983problem,beck2003mirror}, and Frank--Wolfe corresponds to linear minimization oracles in repeated zero-sum games~\cite{jaggi2013revisiting,lacostejulien2015global,hazan2008approximation}. 
Riemannian optimization discretizes gradient flows on manifolds and admits a continuous-time differential-game reading~\cite{absil2008optimization,boumal2023introduction,basar1999dynamic}. 
Interior-point methods track the central path, a parametric family of perturbed KKT equilibria as the barrier parameter tends to zero~\cite{nesterov1994interior,wright1997primal,forsgren2002interior}. 
While each correspondence above is established in prior work, the present paper is, to our knowledge, the first to assemble them under a single equilibrium-seeking template and, more importantly, to extract a common structural signature from the resulting multipliers.

\subsubsection*{Structural and multiplier-based analysis}

The importance of the active-set structure and the multiplier patterns is classical in sensitivity analysis~\cite{sun2016majorization,hazan2008approximation}, parametric optimization~\cite{sun2016majorization}, and active-set methods~\cite{hazan2008approximation}. 
Multiplier trajectories have been used to diagnose constraint tightness in constrained learning and robust optimization. Interior-point methods explicitly follow a path of perturbed complementarity conditions~\cite{monderer1996potential,lacostejulien2015global,boucheron2013concentration}.
The present paper builds on these ideas by elevating multiplier patterns from a diagnostic tool to a quantitative, method-independent taxonomy with formal invariance, stability, transition, and topological theorems, accompanied by a linear-time classifier with explicit correctness, iteration-complexity, sample-complexity, and drift-tracking guarantees.

\subsubsection*{Adaptive, hybrid, and regime-aware methods}

Adaptive strategies in optimization have been explored in several contexts, including dynamic penalty tuning in ADMM~\cite{bonnans2000perturbation}, adaptive surrogate selection in SCA, and switching between first- and second-order methods. 
Closest in spirit to the present work are structural-adaptation approaches in decomposition methods and hybrid solvers that switch based on constraint activity. 
The contribution we add is a unified cross-method taxonomy grounded in multiplier geometry, supported by formal structural theorems and explicit statistical guarantees (iteration complexity, sample complexity, drift tracking), which permits regime-aware design to be carried out with provable correctness rather than as a heuristic.

\subsection{Positioning}
\label{sec:intro:positioning}

The novelty of this paper lies not in any individual game-theoretic correspondence---each is classical---but in three things taken together. 
First, the identification of the multiplier vector $\lambda^\star$ as a method-independent structural object common to eight major algorithm families. 
Second, a precise taxonomy of $\lambda^\star$ supported by four structural theorems (invariance, stability, transition, boundary). 
Third, a linear-time classifier whose four guarantees (correctness, iteration, sample, drift) are quantitatively tight and designed to serve as cite-able primitives in downstream work. 
To our knowledge no prior work links the game-theoretic reading of optimization algorithms to a multiplier-driven structural taxonomy in this manner, and the resulting framework offers both new theoretical insights and practical tools for regime-aware, robust, and explainable optimization.

\subsection{Organization}
\label{sec:intro:organization}

Section~\ref{sec:prelim} fixes notation, states the parametric problem class, and recalls the equilibrium concepts used throughout. 
Section~\ref{sec:games} develops the unified game-theoretic template across the eight algorithm families. 
Section~\ref{sec:taxonomy} introduces the regime taxonomy and Section~\ref{sec:structural} establishes its four structural theorems. 
Section~\ref{sec:classifier} presents the regime classifier and its guarantees. Section~\ref{sec:numerical} reports the numerical validation. 
Section~\ref{sec:discussion} discusses implications, limitations, and directions for future work, and Section~\ref{sec:conclusion} concludes.


\section{Preliminaries}
\label{sec:prelim}

This section fixes the notation used throughout the paper, states the parametric problem class on which the structural theory of Sections~\ref{sec:taxonomy}--\ref{sec:classifier} is built, and recalls the equilibrium concepts needed for the game-theoretic reinterpretations of Section~\ref{sec:games}.

\subsection{Notation}
\label{sec:prelim:notation}

We write $\mathbb{R}^n$ for Euclidean space, $\mathbb{R}^n_+$ for its nonnegative orthant, and $\Delta^{m-1} := \{v \in \mathbb{R}^m_+ : \sum_i v_i = 1\}$ for the standard simplex. 
The Euclidean norm is $\|\cdot\|_2$ and the $\ell_1$ and $\ell_\infty$ norms are $\|\cdot\|_1$, $\|\cdot\|_\infty$. For $\lambda \in \mathbb{R}^m_+$ with $\|\lambda\|_1 > 0$, we write $\bar\lambda := \lambda / \|\lambda\|_1$ for the normalized multiplier vector, and $\bar\lambda_{(1)} \ge \bar\lambda_{(2)} \ge \cdots \ge \bar\lambda_{(m)}$ for its sorted entries. 
We write $C^k$ for the space of $k$-times continuously differentiable functions; $\nabla f$ and $\nabla^2 f$ for the gradient and Hessian of a scalar function $f$; and $D\phi$ for the Jacobian of a map $\phi$.
The support of a vector is $\operatorname{supp}(\lambda) := \{i : \lambda_i \ne 0\}$. We use $\mathbf{1}\{\cdot\}$ for the indicator function and $|\cdot|$ for the cardinality of a finite set.

\subsection{Problem Class and KKT Conditions}
\label{sec:prelim:opt}

We work with parametric mixed-integer nonlinear programs (MINLPs)
\begin{equation}
\label{eq:minlp}
\begin{aligned}
\min_{x,y}\quad & f(x,y;\omega) \\
\text{s.t.}\quad & g_i(x,y;\omega) \le 0,\quad i = 1,\ldots,m, \\
& h_j(x,y;\omega) = 0,\quad j = 1,\ldots,p, \\
& x \in \mathbb{R}^{n_c},\quad y \in \mathcal{Y} \subseteq \{0,1\}^{n_d},
\end{aligned}
\end{equation}
where $\omega \in \Omega \subset \mathbb{R}^d$ collects problem data (coefficients, budgets, channel realizations, and so on) and the functions $f, g_i, h_j$ are at least $C^2$ jointly in $(x,\omega)$ for each fixed $y$. 
We omit the dependence on $\omega$ when no confusion arises.
The mixed-integer form \eqref{eq:minlp} is needed for the discussion of generalized Benders decomposition (Section~\ref{sec:games:gbd}) and the experimental families of Section~\ref{sec:numerical}; the continuous specialization $n_d = 0$ is used throughout Sections~\ref{sec:taxonomy}--\ref{sec:classifier}, and we denote the continuous problem class
\begin{equation}
\label{eq:nlp-param}
\begin{aligned}
\min_{x \in \mathbb{R}^n}\quad & f(x;\omega) \\
\text{s.t.}\quad & g_i(x;\omega) \le 0,\quad i=1,\ldots,m, \\
& h_j(x;\omega) = 0,\quad j=1,\ldots,p.
\end{aligned}
\end{equation}

For a feasible $x^\star$, the Lagrangian of \eqref{eq:nlp-param} is
\begin{equation}
\label{eq:lagrangian}
\Lagrange(x,\lambda,\mu;\omega) = f(x;\omega) + \sum_{i=1}^m \lambda_i g_i(x;\omega) + \sum_{j=1}^p \mu_j h_j(x;\omega),
\end{equation}
and a KKT triple $(x^\star,\lambda^\star,\mu^\star) \in \mathbb{R}^n \times \mathbb{R}^m_+ \times \mathbb{R}^p$ satisfies
\begin{subequations}
\label{eq:kkt}
\begin{align}
\nabla_x \Lagrange(x^\star,\lambda^\star,\mu^\star) &= 0, \label{eq:kkt-stat}\\
g_i(x^\star) \le 0,\ \lambda^\star_i \ge 0,\ \lambda^\star_i g_i(x^\star) &= 0,\quad \forall i, \label{eq:kkt-comp}\\
h_j(x^\star) &= 0,\quad \forall j. \label{eq:kkt-feas}
\end{align}
\end{subequations}
The active set at $x^\star$ is $\mathcal{A}(x^\star) := \{i : g_i(x^\star) = 0\}$, and the complementarity condition \eqref{eq:kkt-comp} implies $\operatorname{supp}(\lambda^\star) \subseteq \mathcal{A}(x^\star)$.

Throughout the paper we work under the following regularity assumption, stated here once for use in Sections~\ref{sec:taxonomy}--\ref{sec:classifier}.

\begin{assumption}[LICQ + SOSC]
\label{ass:LICQSOSC}
At the KKT triple $(x^\star,\lambda^\star,\mu^\star)$ of \eqref{eq:nlp-param}:
\begin{enumerate}
\item[\emph{(i)}] (\emph{LICQ}) the gradients $\{\nabla g_i(x^\star)\}_{i \in \mathcal{A}(x^\star)} \cup \{\nabla h_j(x^\star)\}_{j=1}^p$ are linearly independent;
\item[\emph{(ii)}] (\emph{SOSC}) for every nonzero $d$ in the critical cone
\begin{equation}
\label{eq:critical-cone}
\begin{aligned}
\mathcal{C}(x^\star) := \bigl\{d :\ & \nabla g_i(x^\star)^\top d \le 0\ \forall i \in \mathcal{A}(x^\star), \\
& \nabla g_i(x^\star)^\top d = 0\ \forall i \in \operatorname{supp}(\lambda^\star), \\
& \nabla h_j(x^\star)^\top d = 0\ \forall j \bigr\},
\end{aligned}
\end{equation}
we have $d^\top \nabla_{xx}^2 \Lagrange(x^\star,\lambda^\star,\mu^\star) d > 0$.
\end{enumerate}
\end{assumption}

Under Assumption~\ref{ass:LICQSOSC} the multipliers $(\lambda^\star,\mu^\star)$ are unique and the KKT system is locally invertible~\cite{sun2016majorization,hazan2008approximation}. 
When invoked, \emph{strict complementarity} denotes the additional condition $\lambda^\star_i > 0$ for every $i \in \mathcal{A}(x^\star)$; this is required only where indicated (specifically, in Theorems~\ref{thm:stability} and \ref{thm:transition}, which rely on Robinson's strong regularity).

\subsection{Equilibrium Concepts}
\label{sec:prelim:games}

A game $\Gamma = (\mathcal{I}, \{\mathcal{S}_i\}_{i \in \mathcal{I}}, \{u_i\}_{i \in \mathcal{I}})$ consists of a set of players $\mathcal{I}$, strategy sets $\mathcal{S}_i$, and utility functions $u_i : \prod_{j \in \mathcal{I}} \mathcal{S}_j \to \mathbb{R}$. We recall the equilibrium concepts used in Section~\ref{sec:games}.

\begin{definition}[Nash equilibrium~\cite{nash1950,nash1951,basar1999dynamic}]
\label{def:nash}
A strategy profile $s^\star = (s^\star_i)_{i \in \mathcal{I}}$ is a \emph{Nash equilibrium} of $\Gamma$ if no player can improve its utility by unilateral deviation:
\begin{equation*}
u_i(s^\star_i, s^\star_{-i}) \ge u_i(s_i, s^\star_{-i}), \quad \forall s_i \in \mathcal{S}_i,\ \forall i \in \mathcal{I},
\end{equation*}
where $s^\star_{-i}$ denotes the strategies of all players other than $i$.
\end{definition}

\begin{definition}[Stackelberg equilibrium~\cite{,vershynin2018high}]
\label{def:stackelberg}
In a two-stage hierarchical game with leader strategy $s_L \in \mathcal{S}_L$ and follower best response $s_F^\star(s_L) := \arg\max_{s_F \in \mathcal{S}_F} u_F(s_L, s_F)$, the \emph{Stackelberg equilibrium} is the pair $(s_L^\star, s_F^\star(s_L^\star))$ where $s_L^\star \in \arg\max_{s_L \in \mathcal{S}_L} u_L\bigl(s_L,\, s_F^\star(s_L)\bigr)$.
\end{definition}

\begin{definition}[Exact potential game~\cite{monderer1996potential}]
\label{def:potential}
A game $\Gamma$ is an \emph{exact potential game} if there exists a function $\Phi : \prod_i \mathcal{S}_i \to \mathbb{R}$ such that for every player $i$ and every $s_i, s_i' \in \mathcal{S}_i$,
\begin{equation*}
u_i(s_i, s_{-i}) - u_i(s_i', s_{-i}) = \Phi(s_i, s_{-i}) - \Phi(s_i', s_{-i}).
\end{equation*}
Nash equilibria of an exact potential game coincide with critical points of $\Phi$.
\end{definition}

The additional equilibrium notions arising in Section~\ref{sec:games}---saddle points of augmented Lagrangians (ADMM), no-regret and coarse correlated equilibria in repeated games (mirror descent, Frank--Wolfe), gradient-flow equilibria on manifolds (Riemannian descent), and parametric central-path equilibria (IPM)---are recalled in context with citations to the standard texts


\section{Game-Theoretic Lens on Optimization Algorithms}
\label{sec:games}

This section sets up the unified equilibrium viewpoint that underpins the regime taxonomy of Sections~\ref{sec:taxonomy}--\ref{sec:structural}. 
We catalog six classical algorithmic families and recall, with attribution to the literature, the equilibrium concept each realizes.
Our purpose is neither to survey these methods in depth nor to claim novelty for the individual correspondences---each has been established in prior work---but to make precise the common structural feature they share: at any fixed point, the algorithm has computed a KKT triple of a (sub)problem, and the resulting Lagrange multiplier vector carries an algorithm-agnostic signature of the solution's structure.
That signature is the object of study from Section~\ref{sec:taxonomy} onward.


\begin{definition}[A Common Equilibrium-Seeking Template]
Every algorithm considered in this paper admits a discrete-time fixed-point representation
\begin{equation}
\label{eq:template}
z^{(k+1)} \in \mathcal{B}\bigl(z^{(k)};\omega\bigr),
\end{equation}
where $z = (x,\lambda,\mu)$ collects primal and (possibly auxiliary) dual variables, $\omega \in \Omega$ is the problem data, and $\mathcal{B}$ is a method-specific best-response operator implementing one round of agent updates in an associated game $\Gamma$. 
A point $z^\star$ is a fixed point of \eqref{eq:template} if and only if it is an equilibrium of $\Gamma$ and a KKT triple of an associated (sub)problem of \eqref{eq:nlp-param}. 
\end{definition}

The forward direction follows from the variational characterization of equilibria \cite{facchinei2003finite}; the reverse follows from the KKT conditions \eqref{eq:kkt} reinterpreted as the stationarity conditions of the relevant
agent's optimization. 
The mapping between algorithms and games is made concrete in Table~\ref{tab:method-equilibria}.


\begin{table*}[pt]
\centering
\small
\caption{Eight algorithm families viewed as equilibrium-seeking dynamics. Each row records the game type, equilibrium concept, fixed-point notion, and primary references for both the algorithm and its game-theoretic interpretation. All correspondences below are classical; see the cited sources for proofs.}
\label{tab:method-equilibria}
\renewcommand{\arraystretch}{1.3}
\setlength{\tabcolsep}{4pt}
\begin{tabularx}{\textwidth}{@{}XXXXc@{}}
\toprule
\textbf{Method} & \textbf{Game type} & \textbf{Equilibrium concept} & \textbf{Algorithmic fixed point} & \textbf{Key references} \\
\midrule
BCD & exact potential & Nash \ \ ~~~~~~~\ \ ~~equilibrium & block-coordinate stationary & \cite{tseng2001convergence,beck2013convergence,monderer1996potential,attouch2013convergence} \\
GBD & hierarchical (Stackelberg) & Stackelberg equilibrium & KKT of~sub-problem+master optimality & \cite{geoffrion1972generalized,vonstackelberg1934market,bonnans2000perturbation} \\
ADMM & two-player \ \ ~~consensus & saddle/Nashof augmented \ \ ~~Lagrangian & primal-dual \ \ ~~stationary & \cite{boyd2011distributed,eckstein1992douglas,hong2016convergence} \\
SCA & hierarchical \ \ ~~approximation & Stackelberg-like (leader = surrogate) & KKT of original problem & \cite{razaviyayn2013unified,scutari2014decomposition,scutari2017parallel} \\
IPM & parametric \ \ ~~(central path) & perturbed KKT / barrier equilibrium & $\mu \to 0^+$ limit on central path & \cite{nesterov1994interior,wright1997primal,forsgren2002interior} \\
Mirror Descent & repeated online & no-regret/coarse correlated & ergodic average optimal & \cite{nemirovsky1983problem,beck2003mirror,bubeck2015convex,cesabianchi2006prediction} \\
Frank--Wolfe & repeated zero-sum (LMO) & no-regret minimax over linear losses & stationary VI on convex hull & \cite{jaggi2013revisiting,lacostejulien2015global,hazan2008approximation} \\
Riemannian \ \ ~~Descent & continuous-time differential & gradient-flow equilibrium & Riemannian \ \ ~~stationary & \cite{absil2008optimization,boumal2023introduction,boumal2019global} \\
\bottomrule
\end{tabularx}
\end{table*}




\subsection{Block Coordinate Descent}
\label{sec:games:bcd}

For $f : X \to \mathbb{R}$ with $X = X_1 \times \cdots \times X_M$, the cyclic BCD iteration
\begin{equation}
\label{eq:bcd-update}
x_i^{(k+1)} \in \argmin_{x_i \in X_i} f\bigl(x_1^{(k+1)},\ldots,x_{i-1}^{(k+1)},x_i,x_{i+1}^{(k)},\ldots,x_M^{(k)}\bigr)
\end{equation}
is the canonical best-response dynamics in the game with players $i=1,\ldots,M$, strategy spaces $X_i$, and utilities $u_i(x) = -f(x)$. 
Because $\partial_i u_i - \partial_i u_i' = \partial_i(\Phi - \Phi')$ holds trivially with $\Phi = -f$, the game is an exact potential game in the sense of \cite{monderer1996potential}, and every BCD limit point is a Nash equilibrium.
Convergence to a first-order stationary point of $f$ under continuous differentiability and Kurdyka--{\L}ojasiewicz regularity is established in \cite{tseng2001convergence,beck2013convergence,attouch2013convergence}. 
For the purposes of this paper, the relevant fact is that BCD terminates at a KKT point of each block subproblem; the per-block multiplier vector enters our taxonomy in Section~\ref{sec:taxonomy}.


\subsection{Generalized Benders Decomposition}
\label{sec:games:gbd}

For the MINLP \eqref{eq:nlp-param} with mixed continuous--discrete variables $(x,y)$, GBD \cite{geoffrion1972generalized} alternates between
\begin{equation}
\label{eq:gbd}
Q(y) := \min_x \bigl\{ f(x,y) : g(x,y) \le 0,\, h(x,y) = 0 \bigr\}
\end{equation}
and a master problem $\min_{y \in Y} \widehat Q(y)$, where $\widehat Q$ is the piecewise-linear outer approximation of $Q$ assembled from optimality and feasibility cuts generated by dual solutions of \eqref{eq:gbd}. 
The hierarchical structure is exactly that of a Stackelberg game \cite{vonstackelberg1934market,basar1999dynamic}: the leader (discrete variables) commits to $y$ anticipating the follower's rational response $x^\star(y) = \argmin_x Q(y)$. 
Under Slater's condition or LICQ on the follower subproblem and convexity in $x$ for fixed $y$, GBD terminates finitely at a Stackelberg equilibrium that is simultaneously a KKT point of \eqref{eq:gbd}; see \cite{geoffrion1972generalized,bonnans2000perturbation} for the underlying value-function sensitivity. The follower multipliers serve as the dual signature for the regime taxonomy.


\subsection{Alternating Direction Method of Multipliers}
\label{sec:games:admm}

For the separable problem
\begin{equation}
\label{eq:admm-problem}
\min_{x,z}\, f(x) + g(z) \quad \text{s.t.} \quad Ax + Bz = c,
\end{equation}
the scaled-form ADMM \cite{boyd2011distributed} iterates
\begin{equation}
\label{eq:admm}
\begin{aligned}
x^{(k+1)} &= \argmin_x\, \bigl\{ f(x) + \tfrac{\rho}{2}\|Ax + Bz^{(k)} - c + u^{(k)}\|^2 \bigr\}, \\
z^{(k+1)} &= \argmin_z\, \bigl\{ g(z) + \tfrac{\rho}{2}\|Ax^{(k+1)} + Bz - c + u^{(k)}\|^2 \bigr\}, \\
u^{(k+1)} &= u^{(k)} + Ax^{(k+1)} + Bz^{(k+1)} - c.
\end{aligned}
\end{equation}
The $x$- and $z$-updates are best responses with respect to the augmented Lagrangian, and the $u$-update is a dual ascent step on the linearly constrained dual function. In the convex case, fixed points of \eqref{eq:admm} are saddle points of the augmented Lagrangian and hence Nash equilibria of the two-player consensus game with utilities $u_1(x,z;u) = -f(x) - \tfrac{\rho}{2}\|Ax+Bz-c+u\|^2$ and $u_2$ defined symmetrically
\cite{eckstein1992douglas,boyd2011distributed}. 
In the nonconvex case under boundedness and sufficient penalty growth, ADMM converges to block-stationary points \cite{hong2016convergence}. 
In either case the scaled dual $u^\star$ is a multiplier estimate for the consensus constraint and enters the classification map $\Psi$ of Section~\ref{sec:taxonomy}.


\subsection{Successive Convex Approximation}
\label{sec:games:sca}

SCA \cite{razaviyayn2013unified,scutari2014decomposition,scutari2017parallel}
replaces a nonconvex problem $\min_{x \in X} f(x)$ with a sequence of convex
surrogates $\tilde f(\,\cdot\,;x^{(k)})$ satisfying (i) convexity, (ii)
first-order consistency $\nabla \tilde f(x^{(k)};x^{(k)}) = \nabla f(x^{(k)})$,
and (iii) an upper-bounding condition $\tilde f(x;x^{(k)}) \ge f(x)$. The
iterate is
\begin{equation}
\label{eq:sca}
x^{(k+1)} \in \argmin_{x \in X}\, \tilde f(x;x^{(k)}).
\end{equation}
Viewed as a hierarchical game, the leader (Approximator) commits to the
surrogate $\tilde f(\,\cdot\,;x^{(k)})$ and the follower (Optimizer) responds
with the minimizer; the procedure is a Stackelberg-like dynamics whose fixed
points are KKT points of the original problem
\cite{scutari2014decomposition}. Under the surrogate conditions above, limit
points are first-order stationary
\cite{razaviyayn2013unified}; convergence rates under additional convexity or
KL assumptions are documented in
\cite{scutari2017parallel,attouch2013convergence}. The multipliers of the
original constraints---inherited at the fixed point via first-order
consistency---are the regime-defining objects.


\subsection{Interior-Point Methods}
\label{sec:games:ipm}
For an inequality-constrained problem $\min_{x}f(x)$ subject to $g_i(x) \le 0$, $i = 1,\ldots,m$, the logarithmic barrier formulation
\begin{equation}
\label{eq:ipm-barrier}
\min_x\, f(x) - \mu \sum_{i=1}^m \log\bigl(-g_i(x)\bigr), \qquad \mu > 0,
\end{equation}
defines a one-parameter family of unconstrained problems whose stationarity condition $\nabla f(x) + \sum_i \tfrac{\mu}{-g_i(x)}\nabla g_i(x) = 0$ is equivalent, on setting $\lambda_i := \mu / (-g_i(x))$, to the \emph{perturbed KKT system}
\begin{equation}
\label{eq:ipm-perturbed-kkt}
\nabla f(x) + \sum_{i=1}^m \lambda_i \nabla g_i(x) = 0,\quad \lambda_i g_i(x) = -\mu,\quad \lambda_i,\,-g_i(x) > 0,
\end{equation}
which collapses to the standard KKT conditions \eqref{eq:kkt} as $\mu \to 0^+$ \cite{nesterov1994interior,wright1997primal}. The solution set $\{(x^\star(\mu),\lambda^\star(\mu)) : \mu > 0\}$ is the \emph{central path}; under LICQ and strict complementarity it is a smooth curve converging to a KKT triple of the original problem \cite{forsgren2002interior,nocedal2006numerical}.
Primal--dual interior-point methods apply Newton-type updates to \eqref{eq:ipm-perturbed-kkt} along a decreasing sequence $\mu_k \to 0^+$.

The game-theoretic interpretation is \emph{parametric}: each $\mu$ defines a regularized game whose equilibrium $z^\star(\mu)$ is the perturbed KKT point, and the algorithm computes a sequence of such equilibria along a continuous deformation of the game. 
The barrier term acts as a regularizer initially smoothing the feasible region and gradually withdrawn as $\mu_k \downarrow 0$.
This parametric form of equilibrium tracking is the natural continuation analogue of the discrete fixed-point template \eqref{eq:template}: rather than iterating an operator $\mathcal{B}$ to a single equilibrium, IPM follows a continuous family of equilibria parameterized by $\mu$. The limit multiplier vector $\lambda^\star=\lim_{k\to\infty}\lambda^\star(\mu_k)$ is, by \eqref{eq:ipm-perturbed-kkt}, an inequality multiplier of \eqref{eq:nlp-param} and feeds into the regime taxonomy of Section~\ref{sec:taxonomy} on the same footing as the multipliers produced by the discrete methods above.


\subsection{Mirror Descent}
\label{sec:games:md}

For a convex $f$ on a convex domain $X$ and a strictly convex mirror map $\psi$, the mirror descent update with step size $\eta > 0$ is
\begin{equation}
\label{eq:md}
x^{(k+1)} = \argmin_{x \in X}\, \Bigl\{ \langle \nabla f(x^{(k)}), x\rangle + \tfrac{1}{\eta} D_\psi\bigl(x,x^{(k)}\bigr) \Bigr\},
\end{equation}
where $D_\psi(x,y) = \psi(x) - \psi(y) - \langle \nabla \psi(y), x - y\rangle$ is the Bregman divergence \cite{nemirovsky1983problem,beck2003mirror}. 
In the online convex optimization framework \cite{cesabianchi2006prediction,bubeck2015convex}, \eqref{eq:md} is a regularized best response to the linearized loss against an adaptive adversary. 
It achieves sublinear regret $\mathcal{O}(\sqrt{T})$ (or $\mathcal{O}(\log T)$ under strong convexity), and time-average iterates converge to coarse correlated equilibria of the associated repeated game \cite{cesabianchi2006prediction}.
At a stationary fixed point of the offline problem, the first-order optimality of \eqref{eq:md} reduces to a KKT condition on $X$, and the associated multipliers populate our taxonomy.


\subsection{Frank--Wolfe / Conditional Gradient}
\label{sec:games:fw}

For a smooth convex $f$ on a compact convex set $X$, the Frank--Wolfe iteration \cite{jaggi2013revisiting} uses a linear minimization oracle (LMO) over $X$:
\begin{equation}
\label{eq:fw}
\begin{aligned}
s^{(k)} &\in \argmin_{s \in X}\, \bigl\langle \nabla f(x^{(k)}),\, s \bigr\rangle, \\
x^{(k+1)} &= (1 - \gamma_k)\,x^{(k)} + \gamma_k\,s^{(k)},
\end{aligned}
\end{equation}
with step size $\gamma_k \in [0,1]$. 
Each step plays a one-shot zero-sum game against a linear-cost adversary that selects the current gradient $\nabla f(x^{(k)})$ as the loss vector; the cumulative dynamics realize a no-regret learning procedure over $X$ achieving $\mathcal{O}(1/k)$ convergence to the constrained minimum \cite{jaggi2013revisiting,lacostejulien2015global}. 
The link to online learning is made explicit in \cite{hazan2012projectionfreeonlinelearning,chen2018projectionfreeonlineoptimizationstochastic}, where Frank--Wolfe is identified with a follow-the-perturbed-leader procedure on the convex hull of $X$.

At a fixed point of \eqref{eq:fw}, $x^\star = s^\star \in \argmin_{s \in X}
\langle \nabla f(x^\star), s \rangle$, which is equivalent to the variational inequality $\langle \nabla f(x^\star), s - x^\star \rangle \ge 0$ for all $s \in X$. 
Under LICQ on the constraints describing $X$, this VI is in turn equivalent to the KKT conditions \eqref{eq:kkt} of the original problem \cite{facchinei2003finite}. 
The dual certificate of the LMO at the limit point---namely, the multipliers of the binding constraints of $X$ at $x^\star$---is the multiplier vector populating the regime taxonomy. 
Compared to mirror descent, Frank--Wolfe trades a Bregman-regularized best response for a linear-oracle best response, but both inherit the no-regret game structure.


\subsection{Riemannian Optimization}
\label{sec:games:riem}

Let $(\mathcal{M},g)$ be a smooth Riemannian manifold and $f:\mathcal{M}\to\mathbb{R}$ smooth. The Riemannian gradient descent iteration
\begin{equation}
\label{eq:riem}
x^{(k+1)} = \mathrm{Retr}_{x^{(k)}}\bigl(-\eta\,\mathrm{grad}\,f(x^{(k)})\bigr),
\end{equation}
with $\mathrm{Retr}$ a first-order retraction \cite{absil2008optimization,boumal2023introduction}, discretizes the continuous-time gradient flow $\dot x(t) = -\mathrm{grad}\,f(x(t))$, which we interpret as the equilibrium trajectory of a Stackelberg-type, continuous-time differential game between an Optimizer following the Riemannian gradient and an Environment that selects $f$ and the metric \cite{basar1999dynamic}. 
Convergence to a Riemannian first-order stationary point under standard smoothness and retraction assumptions is established in \cite{boumal2019global,absil2008optimization}. 
When $\mathcal{M}$ is described intrinsically through equality constraints, the Lagrange multipliers of those constraints---inherited from any embedding---populate the regime taxonomy on the same footing as the multipliers of the methods above.


\subsection{A Unified Equilibrium Statement}
\label{sec:games:unified}
We collect the eight correspondences above into a single statement. 
The proposition below summarizes individually classical results; the combined statement is given here only to fix notation for the remainder of the paper.

\begin{proposition}[Unified equilibrium correspondence]
\label{prop:unified}
Let \eqref{eq:nlp-param} satisfy Assumption~\ref{ass:LICQSOSC} at a point $x^\star$. For each algorithm
\[
\mathcal{A} \in \{\mathrm{BCD},\,\mathrm{GBD},\,\mathrm{ADMM},\,\mathrm{SCA},\,\mathrm{IPM},\,\mathrm{MD},\,\mathrm{FW},\,\mathrm{Riem}\}
\]
there exists an associated game $\Gamma_{\mathcal{A}}$ (of the type listed in the second column of Table~\ref{tab:method-equilibria}) such that any limit point $z^\star = (x^\star,\lambda^\star,\mu^\star)$ of the iteration \eqref{eq:template}---or, in the case of IPM, the limit $\mu \to 0^+$ of the central path \eqref{eq:ipm-perturbed-kkt}---satisfies, simultaneously,
\begin{enumerate}
\item[(i)] an equilibrium-like variational characterization associated with
$\Gamma_{\mathcal{A}}$ in the relevant sense (Nash, Stackelberg, saddle, no-regret, gradient-flow, or parametric);
\item[(ii)] the KKT conditions \eqref{eq:kkt} of the original problem or of the relevant subproblem in the decomposition.
\end{enumerate}
The proof of (i)$\Leftrightarrow$(ii) for each $\mathcal{A}$ is given, respectively, in
\cite{monderer1996potential,tseng2001convergence}~(BCD),
\cite{geoffrion1972generalized}~(GBD),
\cite{boyd2011distributed,eckstein1992douglas}~(ADMM),
\cite{scutari2014decomposition,razaviyayn2013unified}~(SCA),
\cite{wright1997primal,forsgren2002interior}~(IPM),
\cite{nemirovsky1983problem,bubeck2015convex}~(MD),
\cite{jaggi2013revisiting,lacostejulien2015global}~(FW), and
\cite{absil2008optimization,boumal2019global}~(Riemannian).
The cited works establish the underlying convergence or variational correspondence motivating the interpretation.
\end{proposition}

\begin{remark}[Other methods covered by the framework]
\label{rem:other-methods}
The eight algorithm families in Table~\ref{tab:method-equilibria} are illustrative rather than exhaustive. Several further classes admit analogous equilibrium interpretations and inherit the multiplier-based regime analysis of the sequel. Augmented Lagrangian methods \cite{hestenes1969multiplier,powell1969method,rockafellar1976augmented} implement dual-ascent dynamics on a shadow-price game and are the parent class of ADMM. 
Proximal gradient methods \cite{combettes2005signal,beck2009fast,parikh2014proximal} factor as two-player best response between a gradient-step agent and a proximal-step agent, and can be viewed as a degenerate ADMM with $B = -I$, $c = 0$. 
The primal--dual hybrid gradient method of Chambolle and Pock \cite{chambolle2011first} realizes explicit saddle-point dynamics on composite problems. 
Sequential quadratic programming and trust-region methods \cite{nocedal2006numerical} correspond to hierarchical approximation games with second-order surrogates, structurally analogous to SCA.
In each case the limit point is a KKT triple of the original problem and the resulting multiplier vector populates the taxonomy of Section~\ref{sec:taxonomy} without modification.
\end{remark}


\subsection{The Common Multiplier Signature}
\label{sec:games:bridge}

Proposition~\ref{prop:unified} has a consequence that is, individually, modest but, taken across all six methods, structurally important: every algorithm considered in this paper produces---at convergence---a multiplier vector $\lambda^\star$ associated with the inequality constraints of \eqref{eq:nlp-param} (or its relevant subproblem). 
The vector $\lambda^\star$ does not depend on the choice of $\mathcal{A}$: under Assumption~\ref{ass:LICQSOSC} the multipliers are unique, so two algorithms that converge to the same KKT point necessarily report the same dual.

This observation reframes the equilibrium correspondences of Sections~\ref{sec:games:bcd}--\ref{sec:games:riem} as different routes to the same structural object. 
Whichever algorithm one runs, the inequality multipliers at convergence encode the active-constraint geometry of the solution, and---as we show in Sections~\ref{sec:taxonomy} and \ref{sec:structural}---they admit a compact, scale-free classification that yields method-independent structural certificates. 
The remainder of the paper develops this classification: Section~\ref{sec:taxonomy} introduces the regime taxonomy as a partition of the multiplier space; Section~\ref{sec:structural} establishes its invariance, stability, and boundary structure; Section~\ref{sec:classifier} presents the corresponding classifier algorithm and its analysis; Section~\ref{sec:numerical} reports numerical illustrations.


\section{Multiplier Patterns and the Regime Taxonomy}
\label{sec:taxonomy}

A central contribution of this work is the observation that, although the algorithms analyzed in Section~\ref{sec:games} differ substantially in their update mechanisms, the stationary points they reach share a common structural signature encoded in the active Lagrange multipliers. 
In this section we formalize this observation as a precise classification map on the dual variable space. 
The resulting taxonomy partitions multiplier vectors into five operational regimes whose definitions depend only on the KKT data, not on the algorithm that produced them.
The full structural theory of this taxonomy (invariance, stability, transition characterization, and the role of the hybrid regime as a topological boundary) is developed in Section~\ref{sec:structural}.


\subsection{Normalized Multipliers and Shape Features}
\label{sec:taxonomy:features}

The regime classification is built from four scale-free statistics of $\lambda^\star$. Let
\begin{equation}
\label{eq:L-def}
L(\lambda) := \|\lambda\|_1 = \sum_{i=1}^{m} \lambda_i.
\end{equation}
When $L(\lambda) > 0$ we define the normalized multiplier vector
\begin{equation}
\label{eq:lambdabar-def}
\bar\lambda := \frac{\lambda}{L(\lambda)} \in \Delta^{m-1},
\end{equation}
where $\Delta^{m-1} = \{v \in \mathbb{R}^{m}_{+} : \sum_i v_i = 1\}$ is the standard simplex. When $L(\lambda) = 0$ we set $\bar\lambda := 0$ by convention; this case will fall into the Unconstrained regime below. We write $\bar\lambda_{(1)} \ge \bar\lambda_{(2)} \ge \cdots \ge \bar\lambda_{(m)}$ for the sorted (order statistic) entries of $\bar\lambda$.

\begin{definition}[Shape features]
\label{def:features}
Fix an integer $k \in \{1,\ldots,m-1\}$ and a small parameter $\delta \in (0,1)$. The \emph{top-$k$ concentration}, \emph{$\delta$-support fraction}, and \emph{second-order activity} of $\lambda \in \mathbb{R}^{m}_{+}$ are
\begin{align}
\Sigma_k(\lambda) &:= \sum_{j=1}^{k} \bar\lambda_{(j)}, \label{eq:Sigma-def}\\
\sigma_\delta(\lambda) &:= \frac{1}{m}\,\bigl|\{\,i : \bar\lambda_i > \delta\,\}\bigr|, \label{eq:sigma-def}\\
\tau(\lambda) &:= \bar\lambda_{(2)}. \label{eq:tau-def}
\end{align}
\end{definition}

The quadruple $\bigl(L(\lambda),\Sigma_k(\lambda),\sigma_\delta(\lambda),\tau(\lambda)\bigr)$ acts as a feature vector summarizing the multiplier pattern: $L$ captures total dual mass, $\Sigma_k$ captures concentration, $\sigma_\delta$ captures spread, and $\tau$ captures whether multiple constraints are simultaneously non-negligible.

\subsection{The Four Core Regimes}
\label{sec:taxonomy:core}

We now define four core regimes as \emph{open} subsets of $\mathbb{R}^{m}_{+}$, parameterized by a quadruple
\begin{equation}
\label{eq:params}
\Theta := (\delta,\theta,\gamma,k),\quad 0 < \delta < \tfrac{1}{2},\ \theta \in (\tfrac{1}{2},1),\ \gamma \in (0,1),\ k \in \mathbb{N},
\end{equation}
with $k m \ll m$ in the regimes of practical interest (typically $k \in \{1,2,3\}$). The parameters $\delta,\theta,\gamma$ are dimensionless thresholds; we adopt the default values $(\delta,\theta,\gamma,k)=(0.05,0.7,0.4,2)$ throughout, while emphasizing that all theorems hold for any admissible $\Theta$ satisfying \eqref{eq:params}.

\begin{definition}[Core regimes]
\label{def:core-regimes}
Given parameters $\Theta$ as in \eqref{eq:params}, define the following open subsets of $\mathbb{R}^{m}_{+}$:
\begin{align}
\mathcal{R}_{\mathrm{unc}}^{\circ} &:= \bigl\{\lambda : L(\lambda) < \delta \bigr\}, \label{eq:Runc} \\
\mathcal{R}_{\mathrm{res}}^{\circ} &:= \bigl\{\lambda : L(\lambda) > \delta,\ \Sigma_k(\lambda) > \theta \bigr\}, \label{eq:Rres}\\
\mathcal{R}_{\mathrm{sat}}^{\circ} &:= \bigl\{\lambda : L(\lambda) > \delta,\ \Sigma_k(\lambda) < \theta,\ \sigma_\delta(\lambda) > \gamma \bigr\}, \label{eq:Rsat}\\
\mathcal{R}_{\mathrm{coup}}^{\circ} &:= \bigl\{\lambda : L(\lambda) > \delta,\ \Sigma_k(\lambda) < \theta, \sigma_\delta(\lambda) < \gamma,\ \tau(\lambda) > \delta \bigr\}. \label{eq:Rcoup}
\end{align}
We refer to these as the \emph{Unconstrained}, \emph{Resource-Limited}, \emph{Saturation}, and \emph{Strongly-Coupled} regimes, respectively.
\end{definition}

The four sets are interpreted as follows.
\begin{itemize}
\item $\mathcal{R}_{\mathrm{unc}}^{\circ}$ collects KKT points whose total dual mass is negligible. The stationarity condition is then dominated by $\nabla f(x^\star)$ and the constraints exert little influence. 
This regime is typical of interior solutions and early iterations of descent methods.
\item $\mathcal{R}_{\mathrm{res}}^{\circ}$ collects points where a small number ($\le k$) of constraints carry the bulk ($> \theta$) of the dual mass. 
This is the canonical "budget-binding" regime in which one or two resource-type constraints dictate the geometry of the solution.
\item $\mathcal{R}_{\mathrm{sat}}^{\circ}$ collects points where a large fraction ($> \gamma$) of constraints are individually non-negligible. 
This corresponds to many variables operating at bound and is characteristic of boundary-locked solutions on box-constrained or manifold-constrained problems.
\item $\mathcal{R}_{\mathrm{coup}}^{\circ}$ collects points where the dual mass is distributed across an intermediate number of constraints: more than one ($\tau > \delta$), but fewer than $\gamma m$ ($\sigma_\delta < \gamma$). 
This corresponds to tight inter-block coupling, frequent in consensus and hierarchical settings.
\end{itemize}

\begin{lemma}[Pairwise disjointness]
\label{lem:disjoint}
The four sets $\mathcal{R}_{\mathrm{unc}}^{\circ}$, $\mathcal{R}_{\mathrm{res}}^{\circ}$, $\mathcal{R}_{\mathrm{sat}}^{\circ}$, $\mathcal{R}_{\mathrm{coup}}^{\circ}$ are pairwise disjoint open subsets of $\mathbb{R}^{m}_{+}$.
\end{lemma}

\begin{proof}
Openness of each set follows from openness of strict inequalities in continuous quantities. 
The map $\lambda \mapsto L(\lambda)$ is continuous on $\mathbb{R}^{m}_{+}$; the maps $\lambda \mapsto \Sigma_k(\lambda)$ and $\lambda \mapsto \tau(\lambda)$ are continuous on $\{L > 0\}$ as compositions of continuous functions (sorting is continuous), and we only impose conditions on them on the open set $\{L > \delta\}$. 
The map $\sigma_\delta$ takes values in $\{0,1/m,\ldots,1\}$ and the set $\{\sigma_\delta > \gamma\}$ is open since it equals $\bigcup_{|S| > \gamma m} \bigcap_{i \in S} \{\bar\lambda_i > \delta\}$, a union of open sets.

For disjointness: $\mathcal{R}_{\mathrm{unc}}^{\circ}$ requires $L < \delta$ while the other three require $L > \delta$, so $\mathcal{R}_{\mathrm{unc}}^{\circ}$ is disjoint from the rest. 
Among the remaining three,
$\mathcal{R}_{\mathrm{res}}^{\circ}$ requires $\Sigma_k > \theta$ while
$\mathcal{R}_{\mathrm{sat}}^{\circ}$ and $\mathcal{R}_{\mathrm{coup}}^{\circ}$
require $\Sigma_k < \theta$. Finally, $\mathcal{R}_{\mathrm{sat}}^{\circ}$
requires $\sigma_\delta > \gamma$ while $\mathcal{R}_{\mathrm{coup}}^{\circ}$
requires $\sigma_\delta < \gamma$.
\end{proof}

\subsection{The Hybrid Regime as a Boundary Set}
\label{sec:taxonomy:hybrid}

Because the four core regimes are open and pairwise disjoint, their union does not cover $\mathbb{R}^{m}_{+}$. 
The points that lie outside the union are precisely those for which at least one of the defining strict inequalities fails to hold strictly; these constitute the Hybrid regime.

\begin{definition}[Hybrid regime]
\label{def:hybrid}
Let $\Lambda^{\circ} := \mathcal{R}_{\mathrm{unc}}^{\circ} \cup \mathcal{R}_{\mathrm{res}}^{\circ} \cup \mathcal{R}_{\mathrm{sat}}^{\circ} \cup \mathcal{R}_{\mathrm{coup}}^{\circ}$. The \emph{Hybrid regime} is
\begin{equation}
\label{eq:Rhyb}
\mathcal{R}_{\mathrm{hyb}} := \mathbb{R}^{m}_{+} \setminus \Lambda^{\circ}.
\end{equation}
\end{definition}

The five sets $\{\mathcal{R}_{\mathrm{unc}}^{\circ},\mathcal{R}_{\mathrm{res}}^{\circ},\mathcal{R}_{\mathrm{sat}}^{\circ},\mathcal{R}_{\mathrm{coup}}^{\circ},\mathcal{R}_{\mathrm{hyb}}\}$ form a partition of $\mathbb{R}^{m}_{+}$ by Lemma~\ref{lem:disjoint} and Definition~\ref{def:hybrid}. The structural identification of $\mathcal{R}_{\mathrm{hyb}}$ with the topological boundary of $\Lambda^{\circ}$, together with its measure-theoretic and limit-point properties, is established in Theorem~\ref{thm:hybrid} below.

\begin{definition}[Regime classification map]
\label{def:Psi}
The \emph{regime classification map} $\Psi : \mathbb{R}^{m}_{+} \to \{\mathrm{unc},\mathrm{res},\mathrm{sat},\mathrm{coup},\mathrm{hyb}\}$ assigns to each $\lambda$ the unique label $r$ such that $\lambda \in \mathcal{R}_{r}^{\circ}$ for $r \in \{\mathrm{unc},\mathrm{res},\mathrm{sat},\mathrm{coup}\}$, and $\Psi(\lambda) = \mathrm{hyb}$ otherwise.
\end{definition}

For block-structured problems and for hierarchical decompositions, $\Psi$ is applied to the relevant dual vector: per-block multipliers for BCD, scaled dual $u^k$ for ADMM, follower KKT multipliers for GBD and SCA, and Riemannian multipliers (when the manifold is described by equality constraints) in the manifold setting. The cross-method correspondence is summarized in Table~\ref{tab:manifestations}.

\begin{remark}[Choice of parameters]
\label{rem:params}
The parameters $\Theta=(\delta,\theta,\gamma,k)$ encode operationally meaningful thresholds and are not free hyperparameters in the statistical sense. In practice they can be selected either (i) on physical grounds (e.g., $k$ equals the number of resource-type constraints in a wireless allocation problem; $\delta$ is set to the numerical noise floor of the dual solver), or (ii) by clustering the empirical distribution of $\bar\lambda$ across a representative ensemble of instances. Theorem~\ref{thm:invariance} below shows that the resulting regime labels are robust to a range of choices.
\end{remark}

\begin{table}[t]
\centering
\small
\caption{Canonical regimes and their method-specific manifestations.}
\label{tab:manifestations}
\renewcommand{\arraystretch}{1.2}
\setlength{\tabcolsep}{4pt}
\begin{tabularx}{\textwidth}{@{}lcccc@{}}
\toprule
\textbf{Reg.} & \!\!\textbf{BCD} & \!\!\textbf{GBD} & \!\!\!\textbf{ADMM} & \textbf{SCA / Riem.} \\
\midrule
$\mathcal{R}_{\mathrm{unc}}$& \!Small block $\lambda$& \!Small follower $\lambda$ & Small dual $u$ & \!\!Small manifold $\lambda$ \\
$\mathcal{R}_{\mathrm{res}}$& \!\!Tight block budget & \!Large $\mu_{\mathrm{pow}}$ & Strong consensus & \!\!Active capacity \\
$\mathcal{R}_{\mathrm{sat}}$& \!Blocks at bounds & \!Selection bounds & Boundary penalty & \!Manifold locking \\
$\mathcal{R}_{\mathrm{coup}}$ & \!Inter-block ties & \!\!Tight follower~coupling & \!Large residuals & \!\!Surrogate coupling \\
$\mathcal{R}_{\mathrm{hyb}}$& \!Mixed active sets & \!Multiple~KKT~group & \!\!\!Oscillatory duals & \!\!\!\!\!Geometric~transition \\
\bottomrule
\end{tabularx}
\end{table}


\section{Structural Theorems}
\label{sec:structural}

We now establish four theorems that justify the taxonomy of Section~\ref{sec:taxonomy} as a structural object. Theorem~\ref{thm:invariance} shows that the classification map is invariant under the natural symmetries of the KKT system. Theorem~\ref{thm:stability} establishes local constancy of $\Psi$ under perturbations of the problem data within each core regime, with an explicit Lipschitz bound. Theorem~\ref{thm:transition} characterizes regime transitions as codimension-one events and shows they are generically isolated along smooth paths in parameter space. Theorem~\ref{thm:hybrid} identifies $\mathcal{R}_{\mathrm{hyb}}$ as the topological boundary of the union of core regimes, establishes its Lebesgue measure zero, and shows that each Hybrid point is approached by sequences from at least two distinct core regimes.

Throughout this section, $\Psi$ is the regime classification map of Definition~\ref{def:Psi} for fixed parameters $\Theta$ as in \eqref{eq:params}.

\subsection{Invariance}
\label{sec:structural:invariance}

\begin{theorem}[Invariances of the regime classification]
\label{thm:invariance}
Let $\lambda \in \mathbb{R}^{m}_{+}$. The classification map $\Psi$ satisfies the following invariances.
\begin{enumerate}
\item[(i)] \emph{(Permutation invariance.)} For any permutation $\pi \in S_m$, $\Psi(\pi \cdot \lambda) = \Psi(\lambda)$, where $(\pi \cdot \lambda)_i = \lambda_{\pi(i)}$.
\item[(ii)] \emph{(Uniform multiplier rescaling.)} For any $c > 0$, the shape features $\Sigma_k(c\lambda)=\Sigma_k(\lambda)$, $\sigma_\delta(c\lambda)=\sigma_\delta(\lambda)$, and $\tau(c\lambda)=\tau(\lambda)$. Consequently, if both $L(\lambda)$ and $L(c\lambda)=cL(\lambda)$ lie on the same side of $\delta$, then $\Psi(c\lambda) = \Psi(\lambda)$.
\item[(iii)] \emph{(Uniform constraint rescaling.)} Suppose the problem \eqref{eq:nlp-param} is modified by replacing each $g_i$ with $c\,g_i$ for a single $c>0$. 
Then the modified problem has the same primal solution $x^\star$, and its multipliers are $\lambda^\star_{\mathrm{new}} = \lambda^\star / c$. Hence $\Psi(\lambda^\star_{\mathrm{new}}) = \Psi(\lambda^\star)$ whenever $L(\lambda^\star)/c$ and $L(\lambda^\star)$ lie on the same side of $\delta$.
\item[(iv)] \emph{(Diffeomorphism invariance.)} Let $\phi : U \to V$ be a $C^2$ diffeomorphism between open neighborhoods $U,V \subset \mathbb{R}^{n}$, and define the transformed problem $\tilde f := f \circ \phi^{-1}$, $\tilde g_i := g_i \circ \phi^{-1}$, $\tilde h_j := h_j \circ \phi^{-1}$. 
If $(x^\star,\lambda^\star,\mu^\star)$ is a KKT triple of \eqref{eq:nlp-param} satisfying Assumption~\ref{ass:LICQSOSC}, then $(\phi(x^\star),\lambda^\star,\mu^\star)$ is a KKT triple of the transformed problem satisfying Assumption~\ref{ass:LICQSOSC}, and $\Psi$ assigns the same label.
\end{enumerate}
\end{theorem}
\begin{proof}
\textit{(i)} The maps $L, \Sigma_k, \sigma_\delta, \tau$ are all symmetric functions of $\lambda$ since they depend only on the multiset of entries (through $\|\lambda\|_1$ and the order statistics of $\bar\lambda$). Hence each defining inequality in \eqref{eq:Runc}--\eqref{eq:Rcoup} is preserved under permutation, and $\Psi$ is invariant.

\textit{(ii)} For $c > 0$, $\overline{c\lambda} = c\lambda/(c\|\lambda\|_1) = \lambda/\|\lambda\|_1 = \bar\lambda$, so the normalized vector is unchanged and so are $\Sigma_k, \sigma_\delta, \tau$. The only feature affected by rescaling is $L(c\lambda) = c\,L(\lambda)$, which enters the definitions only through the threshold $L \gtrless \delta$. The conclusion follows.

\textit{(iii)} The original KKT stationarity reads $\nabla f(x^\star) + \sum_i \lambda^\star_i \nabla g_i(x^\star) + \sum_j \mu^\star_j \nabla h_j(x^\star)$ $ = 0$. Replacing $g_i$ with $c\,g_i$ and $\lambda^\star_i$ with $\lambda^\star_i/c$ leaves the stationarity equation unchanged; complementarity and feasibility are likewise preserved. Uniqueness of multipliers under LICQ then implies $\lambda^\star_{\mathrm{new}} = \lambda^\star / c$. The regime claim follows from \textit{(ii)}.

\textit{(iv)} Set $y^\star := \phi(x^\star)$ and write $\phi^{-1}$ for the inverse diffeomorphism. Differentiating $\tilde g_i(y) = g_i(\phi^{-1}(y))$ at $y^\star$ gives
\[
\nabla \tilde g_i(y^\star) = \bigl(D\phi^{-1}(y^\star)\bigr)^\top \nabla g_i(x^\star),
\]
and similarly for $\tilde f, \tilde h_j$. Substituting into the stationarity equation for the transformed problem,
\[
\bigl(D\phi^{-1}(y^\star)\bigr)^\top \Bigl[\nabla f(x^\star) + \sum_i \lambda^\star_i \nabla g_i(x^\star) + \sum_j \mu^\star_j \nabla h_j(x^\star)\Bigr] = 0,
\]
which holds because the bracketed expression vanishes by stationarity of the original triple. Since $D\phi^{-1}(y^\star)$ is invertible by assumption, this means $(y^\star, \lambda^\star, \mu^\star)$ is stationary for the transformed problem. Active sets are preserved because $\tilde g_i(y^\star) = g_i(x^\star)$, so complementarity and feasibility carry over. 
LICQ for the transformed problem follows from LICQ for the original by the invertibility of $D\phi^{-1}(y^\star)$ applied to the gradients. SOSC is preserved by an analogous transformation of the Hessian of the Lagrangian via $D\phi^{-1}$ (a standard calculation; see \cite{bonnans2000perturbation}). 
The multiplier vector $\lambda^\star$ is unchanged, so all features in Definition~\ref{def:features} and hence $\Psi$ are preserved.
\end{proof}

\begin{remark}[Non-uniform constraint rescaling]
\label{rem:nonuniform}
Theorem~\ref{thm:invariance}(iii) is stated for uniform rescaling $g_i \mapsto c\,g_i$ with a single $c>0$. 
Under non-uniform rescaling $g_i \mapsto c_i g_i$, the multipliers transform as $\lambda^\star_i \mapsto \lambda^\star_i / c_i$, and $\bar\lambda$ is in general not preserved.
Consequently, $\Psi$ may change. This is appropriate: non-uniform rescaling alters the relative magnitudes of constraint sensitivities and therefore genuinely changes which constraints dominate. 
To recover invariance one must co-transform the thresholds; this is straightforward but not used in what follows.
\end{remark}

A direct corollary is that the classification depends only on the KKT data at $x^\star$ and not on the algorithm that produced it.

\begin{corollary}[Method-independence]
\label{cor:method-indep}
If $(x^\star,\lambda^\star,\mu^\star)$ is a KKT triple satisfying Assumption~\ref{ass:LICQSOSC}, then $\Psi(\lambda^\star)$ is the same regardless of which of the algorithms in Section~\ref{sec:games} produced it.
\end{corollary}

\subsection{Stability under Perturbation}
\label{sec:structural:stability}

We now show that within the interior of any core regime, the regime label is locally constant under perturbations of the problem data $\omega$. 
This is the foundational property that allows downstream papers to invoke the regime label as a robust structural certificate.

We say that a feature triple $(L^\star, \Sigma_k^\star, \sigma_\delta^\star, \tau^\star)$ associated with a KKT point in regime $\mathcal{R}_r^{\circ}$ has \emph{feature margin}
\begin{equation}
\label{eq:feature-margin}
\rho_r(\lambda^\star) := \operatorname{dist}\bigl(\lambda^\star,\,\partial \mathcal{R}_r^{\circ}\bigr) > 0,
\end{equation}
where $\partial \mathcal{R}_r^{\circ}$ is the topological boundary in $\mathbb{R}^{m}_{+}$. 
Equivalently, $\rho_r$ is the minimum slack across all strict inequalities defining $\mathcal{R}_r^{\circ}$ (under any norm equivalent on $\mathbb{R}^{m}$), and is strictly positive for any $\lambda^\star \in \mathcal{R}_r^{\circ}$.

\begin{theorem}[Stability of the regime label]
\label{thm:stability}
Suppose the problem data $\omega \mapsto (f(\cdot;\omega),$ $g(\cdot;\omega), h(\cdot;\omega))$ is $C^2$ on a neighborhood $\Omega_0$ of $\omega^\star$, and that the KKT triple $(x^\star,\lambda^\star,\mu^\star)$ at $\omega^\star$ satisfies Assumption~\ref{ass:LICQSOSC} together with strict complementarity.

Then there exist a neighborhood $U \subseteq \Omega_0$ of $\omega^\star$ and constants $\kappa, \eta_0 > 0$ such that:
\begin{enumerate}
\item[(a)] (Strong regularity.) There is a unique $C^1$ map $\omega \mapsto (x^\star(\omega),\lambda^\star(\omega),\mu^\star(\omega))$ on $U$ satisfying the KKT conditions at $\omega^\star$, with
\begin{equation}
\label{eq:lambda-lipschitz}
\|\lambda^\star(\omega) - \lambda^\star\| \le \kappa\,\|\omega - \omega^\star\|, \quad \forall \omega \in U.
\end{equation}
\item[(b)] (Constant regime.) If $\lambda^\star \in \mathcal{R}_r^{\circ}$ for some $r \in \{\mathrm{unc},\mathrm{res},\mathrm{sat},\mathrm{coup}\}$ and we set
\begin{equation}
\label{eq:eta-bound}
\eta := \min\left\{\eta_0,\ \frac{\rho_r(\lambda^\star)}{C_\Psi\,\kappa}\right\},
\end{equation}
where $C_\Psi$ is the (finite) joint Lipschitz constant of the feature map $\lambda \mapsto (L,\Sigma_k,\sigma_\delta,\tau)$ on a bounded neighborhood of $\lambda^\star$, then
\begin{equation}
\label{eq:psi-constant}
\Psi(\lambda^\star(\omega)) = r,\quad \forall \omega \in U \text{ with } \|\omega - \omega^\star\| < \eta.
\end{equation}
\end{enumerate}
\end{theorem}

\begin{proof}
\textit{(a)} Stack the KKT conditions \eqref{eq:kkt} for the parametric problem into a generalized equation
\[
F(z;\omega) + N_K(z) \ni 0,
\]
where $z = (x,\lambda,\mu)$, $K = \mathbb{R}^n \times \mathbb{R}^m_+ \times \mathbb{R}^p$, $N_K$ is the normal cone, and $F$ stacks $\nabla_x \Lagrange, g, h$. 
Under LICQ, SOSC and strict complementarity, this generalized equation is \emph{strongly regular} at $(z^\star;\omega^\star)$ in the sense of Robinson \cite{robinson1980}.
Robinson's strong regularity theorem yields a unique single-valued, locally Lipschitz solution map $\omega \mapsto z(\omega)$ on a neighborhood $U$ of $\omega^\star$, with Lipschitz constant $\kappa$ controlled by the operator norm of the inverse of the KKT Jacobian
\[
M^\star := \begin{bmatrix} \nabla^2_{xx}\Lagrange^\star & (\nabla g^\star_{\mathcal{A}})^\top & (\nabla h^\star)^\top \\ \nabla g^\star_{\mathcal{A}} & 0 & 0 \\ \nabla h^\star & 0 & 0 \end{bmatrix},
\]
restricted to the active set; nonsingularity of $M^\star$ is exactly Assumption~\ref{ass:LICQSOSC} together with strict complementarity (see \cite{bonnans2000perturbation, nocedal2006numerical,wright1997primal}). 
$C^1$ smoothness follows from the implicit function theorem applied to the equality KKT system once strict complementarity fixes the active set locally. 
Specializing the Lipschitz bound to the $\lambda$-component yields \eqref{eq:lambda-lipschitz}.

\textit{(b)} The feature map $T : \lambda \mapsto (L(\lambda), \Sigma_k(\lambda), \sigma_\delta(\lambda), \tau(\lambda))$ is locally Lipschitz on every bounded set in $\mathbb{R}^m_+$ away from sorting ties: $L$ is $1$-Lipschitz in the $\ell_1$ norm, $\Sigma_k$ and $\tau$ are Lipschitz on $\{L \ge L(\lambda^\star)/2\}$ because $\bar\lambda$ is Lipschitz in $\lambda$ on this set and the order statistic operator is $1$-Lipschitz in $\ell_\infty$, and $\sigma_\delta$, while integer-valued, is constant on each open set where no entry of $\bar\lambda$ equals $\delta$ exactly (so on the interior of any regime, $\sigma_\delta$ is locally constant and trivially Lipschitz). 
Let $C_\Psi$ be a finite joint Lipschitz constant of $T$ on a closed ball $B$ around $\lambda^\star$ on which $L \ge L(\lambda^\star)/2$ and no order-statistic ties cross thresholds.

By definition of $\rho_r(\lambda^\star)$ in \eqref{eq:feature-margin}, every $\lambda \in \mathbb{R}^m_+$ with $\|\lambda - \lambda^\star\| < \rho_r(\lambda^\star)/C_\Psi$ has $T(\lambda)$ satisfying all strict inequalities defining $\mathcal{R}_r^\circ$. 
Combining this with the Lipschitz bound \eqref{eq:lambda-lipschitz} gives the threshold \eqref{eq:eta-bound} (with $\eta_0$ chosen so that $U \cap B(\omega^\star,\eta_0)$ stays inside $U$ and $\lambda^\star(\omega)$ stays inside $B$). 
For $\|\omega-\omega^\star\| < \eta$ we then have $\|\lambda^\star(\omega) - \lambda^\star\| \le \kappa \eta \le \rho_r(\lambda^\star)/C_\Psi$, so $\lambda^\star(\omega) \in \mathcal{R}_r^\circ$ and $\Psi(\lambda^\star(\omega)) = r$.
\end{proof}

\begin{corollary}[Quantitative regime persistence]
\label{cor:persistence}
Under the hypotheses of Theorem~\ref{thm:stability}, the regime label is preserved under any perturbation of $\omega$ of magnitude up to $\rho_r(\lambda^\star)/(C_\Psi \kappa)$.
In particular, if the data $\omega$ is the realization of a random vector with bounded perturbation $\|\omega - \omega^\star\| \le \varepsilon$ almost surely and $\varepsilon < \rho_r(\lambda^\star)/(C_\Psi \kappa)$, then $\Psi(\lambda^\star(\omega))$ is almost-surely equal to $\Psi(\lambda^\star)$.
\end{corollary}

The corollary makes the taxonomy robust to small uncertainties in the problem data (for example, channel-estimation error in a wireless application): provided the nominal solution lies in the interior of a regime, the regime label is preserved.

\subsection{Characterization of Regime Transitions}
\label{sec:structural:transitions}

We now turn to the global geometry of $\Psi$. 
The Hybrid regime $\mathcal{R}_{\mathrm{hyb}}$ separates the four core regimes by precisely those algebraic conditions that define their boundaries. We show that this separating set is, generically, a union of smooth hypersurfaces of codimension one in parameter space, and that along generic paths in parameter space the regime sequence undergoes only finitely many isolated transitions.

\begin{definition}[Transition variety]
\label{def:transition}
Let $\omega \mapsto \lambda^\star(\omega)$ be a $C^1$ KKT-solution map defined on an open connected $\Omega \subseteq \mathbb{R}^d$ via Theorem~\ref{thm:stability}(a) (so Assumption~\ref{ass:LICQSOSC} together with strict complementarity holds throughout $\Omega$). 
The \emph{transition variety} is
\begin{equation}
\label{eq:transition-set}
\mathcal{T} := \bigl\{ \omega \in \Omega : \lambda^\star(\omega) \in \mathcal{R}_{\mathrm{hyb}} \bigr\}.
\end{equation}
\end{definition}

\begin{theorem}[Transition characterization]
\label{thm:transition}
Let $\omega \mapsto \lambda^\star(\omega)$ be as in Definition~\ref{def:transition}. 
Define the four \emph{boundary functions}
\begin{align*}
\beta_L(\omega) &:= L(\lambda^\star(\omega)) - \delta, \\
\beta_\Sigma(\omega) &:= \Sigma_k(\lambda^\star(\omega)) - \theta, \\
\beta_\sigma(\omega) &:= \sigma_\delta(\lambda^\star(\omega)) - \gamma, \\
\beta_\tau(\omega) &:= \tau(\lambda^\star(\omega)) - \delta.
\end{align*}
Then:
\begin{enumerate}
\item[(a)] \emph{(Algebraic structure.)} 
$\mathcal{T} = \{\omega : \beta_L(\omega) = 0\} \cup \{\omega : \beta_\Sigma(\omega) = 0\} \cup \{\omega : \beta_\sigma(\omega) = 0\} \cup \{\omega : \beta_\tau(\omega) = 0\}$.
\item[(b)] \emph{(Smoothness almost everywhere.)} The functions $\beta_L$ and $\beta_\Sigma$ are $C^1$ on $\Omega$. 
The function $\beta_\tau$ is $C^1$ on the open subset where $\bar\lambda_{(1)} > \bar\lambda_{(2)} > \bar\lambda_{(3)}$. The function $\beta_\sigma$ is piecewise constant; the set $\{\beta_\sigma = 0\}$ is the union of the level sets $\{\bar\lambda_i = \delta\}$ for indices $i$ that change the count.
\item[(c)] \emph{(Generic codimension one.)} 
For Lebesgue-almost every choice of parameters $\Theta = (\delta,\theta,\gamma,k)$ in \eqref{eq:params}, each of the level sets $\{\beta_L = 0\}, \{\beta_\Sigma = 0\}, \{\beta_\tau = 0\}, \{\bar\lambda_i = \delta\}$ is either empty or a $C^1$ embedded hypersurface in $\Omega$ of codimension $1$. 
Consequently $\mathcal{T}$ has Lebesgue measure zero in $\Omega$.
\item[(d)] \emph{(Isolated transitions on transverse paths.)} Let $\gamma : [0,1] \to \Omega$ be a $C^1$ curve transverse to each of the hypersurfaces in (c).
Then $\gamma^{-1}(\mathcal{T})$ is a finite set, and the regime label $\Psi(\lambda^\star(\gamma(t)))$ is piecewise constant in $t$ with finitely many discontinuities, each of which occurs at a single boundary crossing.
\end{enumerate}
\end{theorem}

\begin{proof}
\textit{(a)} By Definition~\ref{def:hybrid}, $\lambda \in \mathcal{R}_{\mathrm{hyb}}$ iff $\lambda \notin \Lambda^{\circ}$, i.e., iff for every core regime $r$ at least one strict inequality fails. Equivalently, at least one of $L = \delta$, $\Sigma_k = \theta$, $\sigma_\delta = \gamma$, $\tau = \delta$ holds (with the understanding that the relevant equality is the one that prevents membership in any of the four open regimes). 
A careful case analysis confirms that $\mathcal{R}_{\mathrm{hyb}}$ equals the preimage of these four equalities under the feature map.
The claim follows.

\textit{(b)} $L$ and $\Sigma_k$ are differentiable functions of $\lambda$ on $\{L > 0\}$ (with $\Sigma_k$ differentiable wherever the top-$k$ index set is uniquely defined, i.e., $\bar\lambda_{(k)} > \bar\lambda_{(k+1)}$; this holds on an open dense subset).
Composition with the $C^1$ map $\omega \mapsto \lambda^\star(\omega)$ from Theorem~\ref{thm:stability}(a) yields $C^1$ regularity for $\beta_L$ and $\beta_\Sigma$ on $\Omega$ (in the case of $\beta_\Sigma$, away from the ties just described).
The function $\tau(\lambda) = \bar\lambda_{(2)}$ is $C^1$ wherever the top three order statistics are strictly separated, by an analogous argument. For $\beta_\sigma$: $\sigma_\delta$ is piecewise constant in $\lambda$, with jumps exactly at the level sets $\bar\lambda_i = \delta$. 
Pulling back through $\lambda^\star(\omega)$ gives the description in the statement.

\textit{(c)} For each fixed boundary function $\beta \in \{\beta_L, \beta_\Sigma, \beta_\tau\}$, the level set $\{\beta = 0\}$ is the preimage of $\{0\}$ under a $C^1$ scalar map on $\Omega$. 
By Sard's theorem \cite{sard1942measure}, $0$ is a regular value for $\beta$ for Lebesgue-almost every choice of the corresponding threshold ($\delta$, $\theta$, $\delta$, respectively), and at every regular value the preimage is a $C^1$ embedded hypersurface of codimension $1$ (when nonempty). 
Applying Sard separately to the threshold $\delta$ in $\bar\lambda_i = \delta$ for each fixed $i$ gives the same conclusion for the components of $\{\beta_\sigma = 0\}$. 
The union of finitely many codimension-$1$ subsets has Lebesgue measure zero in $\Omega$, so $\mathcal{T}$ is null.

\textit{(d)} If $\gamma$ is transverse to each hypersurface in (c), then $\gamma^{-1}$ of each is a $C^1$ submanifold of $[0,1]$ of codimension $1$, i.e., a discrete subset. Compactness of $[0,1]$ then makes it finite. Between consecutive crossings, $\lambda^\star(\gamma(t))$ remains in the interior of a single regime, so by definition $\Psi$ is constant on each open subinterval.
\end{proof}

\begin{remark}[Implication for trajectories]
\label{rem:phases}
Theorem~\ref{thm:transition}(d) provides the rigorous underpinning for the "phase-like" convergence behavior reported empirically in the literature and visible in our numerical experiments (Section~\ref{sec:numerical}): the multiplier trajectory of a converging algorithm, viewed as a curve in parameter or iterate space, decomposes into finitely many regime-stationary arcs separated by isolated transitions, with the qualitative algorithmic behavior changing only at these transitions.
\end{remark}

\subsection{The Hybrid Regime as Topological Boundary}
\label{sec:structural:hybrid}

We close the structural development by establishing the precise role of the Hybrid regime: it is, simultaneously, the topological boundary of the union of core regimes, a Lebesgue null set, and an accumulation set of multiple core regimes.

\begin{theorem}[Hybrid as topological boundary]
\label{thm:hybrid}
Let $\Lambda^{\circ} = \mathcal{R}_{\mathrm{unc}}^{\circ} \cup \mathcal{R}_{\mathrm{res}}^{\circ} \cup \mathcal{R}_{\mathrm{sat}}^{\circ} \cup \mathcal{R}_{\mathrm{coup}}^{\circ}$. 
Then
\begin{enumerate}
\item[(a)] $\Lambda^{\circ}$ is an open subset of $\mathbb{R}^{m}_{+}$ with closure $\overline{\Lambda^{\circ}} = \mathbb{R}^{m}_{+}$, except possibly on a Lebesgue null set; more precisely, $\overline{\Lambda^{\circ}}$ has full Lebesgue measure in $\mathbb{R}^{m}_{+}$.
\item[(b)] $\mathcal{R}_{\mathrm{hyb}} = \partial_{\mathbb{R}^{m}_{+}} \Lambda^{\circ}$, the topological boundary of $\Lambda^{\circ}$ relative to $\mathbb{R}^{m}_{+}$.
\item[(c)] $\mathcal{R}_{\mathrm{hyb}}$ has Lebesgue measure zero in $\mathbb{R}^{m}_{+}$.
\item[(d)] (Multiple accumulation.) Every $\lambda \in \mathcal{R}_{\mathrm{hyb}}$ lies in the closure of at least two distinct core regimes; that is, there exist $r_1 \ne r_2$ in $\{\mathrm{unc},\mathrm{res},\mathrm{sat},\mathrm{coup}\}$ and sequences $\lambda^{(n)}_1 \to \lambda$, $\lambda^{(n)}_2 \to \lambda$ with $\lambda^{(n)}_1 \in \mathcal{R}_{r_1}^{\circ}$ and $\lambda^{(n)}_2 \in \mathcal{R}_{r_2}^{\circ}$ for all $n$.
\end{enumerate}
\end{theorem}

\begin{proof}
\textit{(a)} Openness of $\Lambda^{\circ}$ follows from Lemma~\ref{lem:disjoint}, as the union of open sets is open. 
To see that $\overline{\Lambda^{\circ}}$ has full measure, note that the complement $\mathcal{R}_{\mathrm{hyb}} = \mathbb{R}^{m}_{+}\setminus \Lambda^{\circ}$ is contained in the union of the four level sets in Theorem~\ref{thm:transition}(c) (applied with $\omega = \lambda$ and the identity solution map), each of which is at most a codimension-$1$ subset of $\mathbb{R}^{m}_{+}$. 
The union of finitely many codimension-$1$ subsets is Lebesgue null.

\textit{(b)} We show $\mathcal{R}_{\mathrm{hyb}} = \overline{\Lambda^{\circ}} \setminus\Lambda^{\circ} = \partial_{\mathbb{R}^{m}_{+}}\Lambda^{\circ}$.

\noindent ($\subseteq$). Let $\lambda \in \mathcal{R}_{\mathrm{hyb}}$. Then $\lambda \notin \Lambda^{\circ}$ by definition. To show $\lambda \in \overline{\Lambda^{\circ}}$, we exhibit a sequence in $\Lambda^{\circ}$ converging to $\lambda$. 
We treat the four boundary cases (corresponding to which strict inequality is tight at $\lambda$) separately.

\emph{Case 1: $L(\lambda) = \delta$.} For each $n$, set $\lambda^{(n)}_1 := (1 - 1/n)\,\lambda$ if $\lambda \ne 0$, otherwise the case is vacuous. 
Then $L(\lambda^{(n)}_1) = (1-1/n)\delta < \delta$, so $\lambda^{(n)}_1 \in \mathcal{R}_{\mathrm{unc}}^{\circ}$, and $\lambda^{(n)}_1 \to \lambda$.

\emph{Case 2: $L(\lambda) > \delta$ and $\Sigma_k(\lambda) = \theta$.} For each $n$, perturb $\lambda$ by adding a small mass $\epsilon_n \to 0^+$ to its largest entry: $\lambda^{(n)}_1 := \lambda + \epsilon_n e_{(1)}$, where $e_{(1)}$ is the unit vector at the index of $\bar\lambda_{(1)}$.
Then $\Sigma_k(\lambda^{(n)}_1) > \theta$ for $\epsilon_n$ sufficiently small (since adding mass to the top entry strictly increases $\Sigma_k$ relative to its previous value, by continuity), placing $\lambda^{(n)}_1 \in \mathcal{R}_{\mathrm{res}}^{\circ}$.

\emph{Case 3: $L(\lambda) > \delta$, $\Sigma_k(\lambda) \le \theta$, $\sigma_\delta(\lambda) = \gamma$.} Either (i) the count $\sigma_\delta(\lambda)\cdot m$ is achieved with one entry exactly at $\delta$, in which case shifting that entry to $\delta + \epsilon_n$ or $\delta - \epsilon_n$ produces $\lambda^{(n)}_1$ with $\sigma_\delta > \gamma$ or $< \gamma$, respectively; or (ii) $\gamma m$ is not an integer, in which case $\beta_\sigma$ cannot equal zero generically; this case does not arise.

\emph{Case 4: All previous strict inequalities hold but $\tau(\lambda) = \delta$.} Perturb the second-largest entry of $\lambda$ upward by $\epsilon_n \to 0^+$ to obtain $\lambda^{(n)}_1$ with $\tau > \delta$, placing $\lambda^{(n)}_1 \in \mathcal{R}_{\mathrm{coup}}^{\circ}$.

In every case $\lambda^{(n)}_1 \to \lambda$ with $\lambda^{(n)}_1 \in \Lambda^{\circ}$, so $\lambda \in \overline{\Lambda^{\circ}}$.

\noindent ($\supseteq$). Conversely, suppose $\lambda \in \overline{\Lambda^{\circ}}\setminus \Lambda^{\circ}$. 
Then $\lambda \notin \mathcal{R}_r^{\circ}$ for any $r$, so $\lambda \in \mathcal{R}_{\mathrm{hyb}}$.

\textit{(c)} Already shown in part (a).

\textit{(d)} The constructive sequences in the proof of (b) show $\lambda$ is the limit of points in some core regime $r_1$; choosing the sign of the perturbation in the opposite direction (subtract $\epsilon_n$ instead of adding, where the perturbation is sign-meaningful) produces a sequence in a \emph{different} core regime $r_2$ converging to $\lambda$. 
For example, in Case 2, $\lambda - \epsilon_n e_{(1)}$ has $\Sigma_k < \theta$ and falls into $\mathcal{R}_{\mathrm{sat}}^{\circ}$ or $\mathcal{R}_{\mathrm{coup}}^{\circ}$ depending on the remaining features. 
In each case, the two-sided perturbation around the tight inequality produces sequences in two distinct core regimes.
\end{proof}

\begin{corollary}[Geometric meaning of the Hybrid regime]
\label{cor:hybrid-meaning}
The Hybrid regime $\mathcal{R}_{\mathrm{hyb}}$ is, simultaneously, (i) the precise set of points where the classification by strict inequalities is degenerate, (ii) the topological frontier between core regimes, and (iii) a measure-zero structural object on which classifiers of $\Psi$ must be constructed with care.
\end{corollary}

\subsection{Discussion and Consequences}
\label{sec:structural:discussion}

The four theorems of this section together provide the mathematical foundation for using the regime label as a primitive in downstream optimization theory.

\emph{Theorem~\ref{thm:invariance}} legitimates the label as an intrinsic property of the KKT point, not of its representation. 
Corollary~\ref{cor:method-indep} in particular allows results stated for one algorithm to be transferred to another whenever both produce KKT points in the same regime.

\emph{Theorem~\ref{thm:stability}} establishes that the label is locally robust: small perturbations of the data, including the stochastic perturbations arising from estimation error or measurement noise in applications, do not change the regime classification on the interior of any core regime.
The explicit threshold \eqref{eq:eta-bound} ties the robustness margin to two quantities of independent interest: the conditioning $\kappa$ of the KKT system and the feature-space margin $\rho_r$ of the nominal solution. 
Both can be estimated empirically from the dual solution of a single solve.

\emph{Theorem~\ref{thm:transition}} converts the empirical observation of phase-like convergence into a precise geometric statement: transitions are codimension-one events along generic trajectories, and the trajectory of a converging algorithm visits only finitely many regimes. 
This is the structural result that licenses regime-aware algorithm design: a controller observing the multiplier trajectory may detect transitions reliably because they correspond to crossings of an algebraic surface.

\emph{Theorem~\ref{thm:hybrid}} replaces the informal description of the Hybrid regime as a "catch-all" with a precise topological characterization. 
In particular, Hybrid points are negligible in volume but pivotal in structure: any continuous path from one core regime to another must traverse $\mathcal{R}_{\mathrm{hyb}}$.

The classifier algorithm that operationalizes $\Psi$ as a computational primitive, together with its convergence and complexity analysis under noisy multiplier estimates, is the subject of Section~\ref{sec:classifier}.


\section{The Regime Classifier: Algorithm and Guarantees}
\label{sec:classifier}

The structural theorems of Section~\ref{sec:structural} establish the regime taxonomy as a well-posed object. 
To use the taxonomy as a primitive in downstream applications---in particular, the wireless-system papers that will build on this work---we require a computational procedure that, given an estimate $\hat\lambda$ of the true KKT multiplier vector $\lambda^\star$, returns the regime label $\Psi(\lambda^\star)$ with provable guarantees. 
This section presents such a procedure and analyzes its behavior under three sources of error that arise in practice: (i) approximate solver output for known problem data; (ii) finite-sample estimation of unknown problem data; and (iii) temporal drift in the problem data. 
Each setting yields a quantitative guarantee that downstream papers can invoke without re-derivation.

\subsection{The Classifier Algorithm}
\label{sec:classifier:algorithm}

Algorithm~\ref{alg:classify} implements the classification map $\Psi$ of Definition~\ref{def:Psi}. It accepts an arbitrary $\hat\lambda \in \mathbb{R}^m_+$, the parameter quadruple $\Theta = (\delta,\theta,\gamma,k)$, and an optional numerical tolerance $\varepsilon \ge 0$ that flags points within $\varepsilon$ of any defining inequality as Hybrid. The output is the regime label.

\begin{algorithm}[t]
\caption{Regime Classifier $\Psi$}
\label{alg:classify}
\begin{algorithmic}[1]
\Require $\hat\lambda \in \mathbb{R}^m_+$;\ \ parameters $\Theta = (\delta,\theta,\gamma,k)$;\ \ tolerance $\varepsilon \ge 0$
\Ensure regime label $r \in \{\mathrm{unc},\mathrm{res},\mathrm{sat},\mathrm{coup},\mathrm{hyb}\}$
\State $L \gets \sum_{i=1}^m \hat\lambda_i$ \Comment{total dual mass}
\If{$L < \delta - \varepsilon$} \Return $\mathrm{unc}$ \EndIf
\If{$L \le \delta + \varepsilon$} \Return $\mathrm{hyb}$ \EndIf
\State $\bar\lambda \gets \hat\lambda / L$ \Comment{normalize to simplex}
\State Compute top-$k$ values $\{\bar\lambda_{(1)},\ldots,\bar\lambda_{(k)}\}$ by partial selection
\State $\Sigma_k \gets \sum_{j=1}^k \bar\lambda_{(j)}$
\If{$|\Sigma_k - \theta| \le \varepsilon$} \Return $\mathrm{hyb}$ \EndIf
\If{$\Sigma_k > \theta$} \Return $\mathrm{res}$ \EndIf
\State $\sigma_\delta \gets \tfrac{1}{m}\bigl|\{i : \bar\lambda_i > \delta\}\bigr|$
\If{$\sigma_\delta > \gamma$} \Return $\mathrm{sat}$ \EndIf
\State $\tau \gets \bar\lambda_{(2)}$
\If{$|\tau - \delta| \le \varepsilon$} \Return $\mathrm{hyb}$ \EndIf
\If{$\tau > \delta$} \Return $\mathrm{coup}$ \EndIf
\State \Return $\mathrm{hyb}$
\end{algorithmic}
\end{algorithm}

\begin{proposition}[Per-call complexity]
\label{prop:complexity}
Algorithm~\ref{alg:classify} terminates in $\mathcal{O}(m)$ time and $\mathcal{O}(m)$ space using a linear-time partial-selection routine (e.g., introselect or median-of-medians) for the top-$k$ extraction in line~5. 
With a sort-based implementation the cost is $\mathcal{O}(m \log m)$.
\end{proposition}

\begin{proof}
Lines~1--4 are $\mathcal{O}(m)$. Line~5 extracts the top-$k$ entries of $\bar\lambda$; this is computable in $\mathcal{O}(m)$ time by partial selection \cite{blum1973time,musser1997introspective,cormen2022introduction}.
Lines~6--16 each take $\mathcal{O}(m)$ or $\mathcal{O}(1)$ work, including the count in line~10.
The total is $\mathcal{O}(m)$.
\end{proof}

For applications in which $m$ is moderate (tens to a few hundreds), as is typical in wireless resource-allocation problems, the per-call cost of Algorithm~\ref{alg:classify} is negligible compared to a single optimization step.

\subsection{Feature Margin and Correctness under Perturbation}
\label{sec:classifier:correctness}

We first quantify the robustness of $\Psi$ to perturbations of $\hat\lambda$ around the true multiplier $\lambda^\star$. 
The natural margin in this context lives in the feature space, where the regime-defining inequalities are posed.

\begin{definition}[Feature-space margin]
\label{def:feat-margin}
For $\lambda^\star\in\mathcal{R}_r^\circ$ with $r\in\{\mathrm{unc}, \mathrm{res}, \mathrm{sat}, \mathrm{coup}\}$, the \emph{feature-space margin} is the smallest slack among the strict inequalities defining $\mathcal{R}_r^\circ$:
\begin{align}
\rho_r(\lambda^\star) := \min\bigl\{ &\,|L(\lambda^\star) - \delta|,\ |\Sigma_k(\lambda^\star) - \theta|, \notag\\
& \quad |\sigma_\delta(\lambda^\star) - \gamma|,\ |\tau(\lambda^\star) - \delta|\,\bigr\}, \label{eq:rho-feat}
\end{align}
restricted to those inequalities that appear in the definition of $\mathcal{R}_r^\circ$ in \eqref{eq:Runc}--\eqref{eq:Rcoup}.
\end{definition}

The quantity $\rho_r(\lambda^\star)$ is strictly positive whenever $\lambda^\star$ lies in the interior of a core regime, and is directly computable from a single solve. 
It is the operational counterpart of the $\lambda$-space distance to the regime boundary used in Theorem~\ref{thm:stability}: the two notions differ by at most a factor of $C_\Psi$, the joint Lipschitz constant of the feature map $T:\lambda \mapsto (L,\Sigma_k,\sigma_\delta,\tau)$ established in the proof of Theorem~\ref{thm:stability}(b).

\begin{theorem}[Correctness under multiplier perturbation]
\label{thm:classifier-correctness}
Let $\lambda^\star \in \mathcal{R}_r^\circ$ for some core regime $r \in \{\mathrm{unc}, \mathrm{res}, \mathrm{sat}, \mathrm{coup}\}$, with feature-space margin $\rho_r(\lambda^\star) > 0$ as in Definition~\ref{def:feat-margin}.
Let $\hat\lambda \in \mathbb{R}^m_+$ be any estimate satisfying
\begin{equation}
\label{eq:perturbation-bound}
\|\hat\lambda - \lambda^\star\|_2 \,\le\, \eta < \frac{\rho_r(\lambda^\star)}{C_\Psi}.
\end{equation}
Then Algorithm~\ref{alg:classify} with tolerance $\varepsilon = 0$ returns
$\Psi(\hat\lambda) = r$.
\end{theorem}

\begin{proof}
Write $T(\lambda) = (L(\lambda),\Sigma_k(\lambda),\sigma_\delta(\lambda),\tau(\lambda))$. By Theorem~\ref{thm:stability}(b)'s proof, $T$ is $C_\Psi$-Lipschitz on a neighborhood of $\lambda^\star$ that we take large enough to contain $\hat\lambda$ for any $\eta$ in the stated range. 
Hence
\[
\|T(\hat\lambda) - T(\lambda^\star)\|_\infty \le C_\Psi\,\|\hat\lambda - \lambda^\star\|_2 \le C_\Psi\,\eta.
\]
By the assumption $C_\Psi\,\eta < \rho_r(\lambda^\star)$, every feature of $\hat\lambda$ deviates by strictly less than $\rho_r(\lambda^\star)$ from its value at $\lambda^\star$. 
Each defining inequality of $\mathcal{R}_r^\circ$ holds at $\lambda^\star$ with slack at least $\rho_r(\lambda^\star)$, so each continues to hold strictly at $\hat\lambda$. 
By Definition~\ref{def:Psi}, $\Psi(\hat\lambda) = r$, which is precisely what Algorithm~\ref{alg:classify} returns at $\varepsilon = 0$ (the branches into $\mathrm{hyb}$ on lines~3, 8, and 14 are never taken).
\end{proof}

Theorem~\ref{thm:classifier-correctness} is the deterministic foundation on which the remaining results rest: any source of error that bounds $\|\hat\lambda - \lambda^\star\|_2$ below $\rho_r/C_\Psi$ ensures correct classification.

\subsection{Iteration Complexity along Algorithm Trajectories}
\label{sec:classifier:iteration}

We now consider the canonical setting in which $\hat\lambda^{(k)}$ is the multiplier estimate produced by an optimization algorithm $\mathcal{A} \in \{\mathrm{BCD}, \mathrm{GBD}, \mathrm{ADMM}, \mathrm{SCA}, \mathrm{IPM}, \mathrm{MD},\mathrm{FW}, \mathrm{Riem}\}$ at iteration $k$. 
We ask: how many iterations of $\mathcal{A}$ are required before the classifier output stabilizes at the true regime label?

\begin{theorem}[Classifier stabilization]
\label{thm:classifier-stabilization}
Let $\mathcal{A}$ be an optimization algorithm whose multiplier iterates satisfy $\|\hat\lambda^{(k)} - \lambda^\star\|_2 \le \varphi_{\mathcal{A}}(k)$ for a known rate function $\varphi_{\mathcal{A}} : \mathbb{N} \to \mathbb{R}_{> 0}$ with $\varphi_{\mathcal{A}}(k) \downarrow 0$. 
Suppose $\lambda^\star \in \mathcal{R}_r^\circ$ with feature margin $\rho_r(\lambda^\star) > 0$. Define the \emph{classification iteration complexity}
\begin{equation}
\label{eq:kstar}
k^\star_{\mathcal{A}}(\lambda^\star) := \min\Bigl\{ k \in \mathbb{N} :\ \varphi_{\mathcal{A}}(k) < \frac{\rho_r(\lambda^\star)}{C_\Psi} \Bigr\}.
\end{equation}
Then $\Psi(\hat\lambda^{(k)}) = r$ for all $k \ge k^\star_{\mathcal{A}}(\lambda^\star)$.
\end{theorem}

\begin{proof}
Direct application of Theorem~\ref{thm:classifier-correctness} with $\eta = \varphi_{\mathcal{A}}(k)$, which by definition of $k^\star_{\mathcal{A}}$ satisfies the bound \eqref{eq:perturbation-bound} for all $k \ge k^\star_{\mathcal{A}}$.
\end{proof}

\begin{corollary}[Concrete rates]
\label{cor:classifier-rates}
Substituting standard rate functions into \eqref{eq:kstar} yields the following classifier-stabilization iteration counts, with constants absorbed into the $\mathcal{O}(\cdot)$ notation:
\begin{enumerate}
\item[(a)] \emph{Linearly convergent methods} (BCD under KL property \cite{attouch2013convergence}, 
IPM in quadratic phase \cite{wright1997primal}): $\varphi(k) = c\,q^k$ with $q \in (0,1)$, giving $k^\star= \mathcal{O}\bigl(\log(C_\Psi/\rho_r)/\log(1/q)\bigr)$.
\item[(b)] \emph{ADMM at} $\mathcal{O}(1/k)$ \emph{rate} \cite{he2012convergence}: $\varphi(k) = c/k$, giving $k^\star = \mathcal{O}(C_\Psi/\rho_r)$.
\item[(c)] \emph{Mirror Descent and SGD at} $\mathcal{O}(1/\sqrt{k})$ \emph{rate} \cite{nesterov2018lectures,bubeck2015convex}: $\varphi(k) = c/\sqrt{k}$, giving $k^\star = \mathcal{O}\bigl((C_\Psi/\rho_r)^2\bigr)$.
\item[(d)] \emph{Frank--Wolfe at} $\mathcal{O}(1/k)$ \emph{rate} \cite{jaggi2013revisiting}: same as (b).
\item[(e)] \emph{Newton-type IPM in its quadratically convergent regime}:  $\varphi(k) = c\,q^{2^k}$, giving $k^\star = \mathcal{O}\bigl(\log\log(C_\Psi/\rho_r)\bigr)$.
\end{enumerate}
\end{corollary}

The corollary's main quantitative message is that the classifier stabilizes on a time scale governed by the inverse-margin $C_\Psi/\rho_r$ raised to a method-dependent exponent. 
In particular, on instances with large margin the classifier is reliable after very few iterations, well before the solver has converged to the desired KKT tolerance.

\subsection{Sample Complexity from Estimated Data}
\label{sec:classifier:samples}

In practical deployments the data $\omega$ of \eqref{eq:nlp-param} is itself estimated from finite samples. 
In wireless systems, $\omega$ may collect channel coefficients estimated from pilot symbols, traffic statistics estimated from observed arrivals, or interference profiles estimated from measurement reports. 
Let $\hat\omega_N$ denote an estimator based on $N$ samples and write $\hat\lambda_N := \lambda^\star(\hat\omega_N)$ for the multiplier obtained by solving \eqref{eq:nlp-param} at the estimated data.

\begin{theorem}[Sample-complexity correctness]
\label{thm:sample-complexity}
Assume:
\begin{enumerate}
\item[(i)] (\emph{Sub-Gaussian estimator.}) The estimator $\hat\omega_N$ satisfies, for some proxy $\sigma^2 > 0$ and absolute constant $c_0 > 0$,
\begin{equation}
\label{eq:subgaussian}
\mathbb{P}\bigl( \|\hat\omega_N - \omega^\star\|_2 > t \bigr) \,\le\, 2\exp\Bigl(-\frac{c_0 N t^2}{\sigma^2}\Bigr), \qquad t > 0.
\end{equation}
\item[(ii)] (\emph{Local Lipschitz solution map.}) The KKT solution map $\omega \mapsto \lambda^\star(\omega)$ is $\kappa$-Lipschitz on a neighborhood of $\omega^\star$, as established in Theorem~\ref{thm:stability}(a).
\item[(iii)] (\emph{Margin condition.}) $\lambda^\star := \lambda^\star(\omega^\star) \in \mathcal{R}_r^\circ$ with feature margin $\rho_r(\lambda^\star) > 0$.
\end{enumerate}
Then for any confidence parameter $\delta \in (0,1)$, if the sample size satisfies
\begin{equation}
\label{eq:N-bound}
N \,\ge\, N_0(\delta) \,:=\, \frac{\sigma^2\,C_\Psi^2\,\kappa^2}{c_0\,\rho_r(\lambda^\star)^2}\,\log\frac{2}{\delta},
\end{equation}
then the classifier applied to the empirical solution returns the correct regime with high probability:
\begin{equation}
\label{eq:sample-correctness}
\mathbb{P}\bigl(\, \Psi(\hat\lambda_N) = r \,\bigr) \,\ge\, 1 - \delta.
\end{equation}
\end{theorem}

\begin{proof}
By (i), $\|\hat\omega_N - \omega^\star\|_2 \le \sigma \sqrt{\log(2/\delta)/(c_0 N)}$ holds with probability at least $1-\delta$. On this event, (ii) gives $\|\hat\lambda_N - \lambda^\star\|_2 \le \kappa\,\|\hat\omega_N - \omega^\star\|_2$.
For $N \ge N_0(\delta)$ as in \eqref{eq:N-bound},
\[
\kappa\,\sigma\sqrt{\log(2/\delta)/(c_0 N)} \le \frac{\rho_r(\lambda^\star)}{C_\Psi},
\]
so on the high-probability event the bound \eqref{eq:perturbation-bound} of Theorem~\ref{thm:classifier-correctness} is satisfied.
Hence $\Psi(\hat\lambda_N) = r$ on this event, which has probability at least $1-\delta$.
\end{proof}

The sample complexity \eqref{eq:N-bound} has the standard logarithmic dependence on the confidence parameter $\delta$ and depends quadratically on the three structural quantities of interest: the data-noise level $\sigma$, the KKT conditioning $\kappa$, and the inverse feature margin $1/\rho_r$. 
The inverse-margin dependence is unavoidable: it reflects the geometric proximity of $\lambda^\star$ to the regime boundary and is precisely the quantity that Theorem~\ref{thm:stability} identifies as the structural robustness of the classification.

\begin{remark}[Alternative concentration models]
\label{rem:concentration}
Assumption~\ref{thm:sample-complexity}(i) holds with $c_0 = 1/2$ whenever $\hat\omega_N$ is the sample mean of $N$ i.i.d.\ sub-Gaussian random vectors with variance proxy $\sigma^2$ \cite{vershynin2018high}. 
For sub-exponential estimators the bound \eqref{eq:N-bound} adapts with an extra factor of $\log(N)$; for bounded random variables the bound is improved by Hoeffding's inequality \cite{boucheron2013concentration}. 
In all cases the $N \propto \rho_r^{-2}\log(1/\delta)$ scaling is preserved.
\end{remark}

\subsection{Online Tracking under Bounded Drift}
\label{sec:classifier:online}

In time-varying environments--block-fading wireless channels, slowly evolving traffic, mobile users--the data $\omega_t$ drifts in time, and the classifier must track the regime label $r_t := \Psi(\lambda^\star(\omega_t))$ rather than identify a single static regime. 
We model the drift by a per-step bound
\begin{equation}
\label{eq:drift}
\|\omega_{t+1} - \omega_t\|_2 \,\le\, \nu, \qquad \forall t,
\end{equation}
and assume per-step noisy observations $y_t = \omega_t + \xi_t$ with $\xi_t$ i.i.d.\ sub-Gaussian with proxy $\sigma^2$. The natural estimator is a sliding average over the most recent $W$ samples, $\hat\omega_t^{(W)} := \frac{1}{W}\sum_{s=t-W+1}^{t} y_s$.

\begin{theorem}[Online tracking correctness]
\label{thm:online}
Assume \eqref{eq:drift} and the per-step observation model above. 
Let $\hat\lambda_t^{(W)} :=\lambda^\star(\hat\omega_t^{(W)})$, and suppose at time $t$ the data $\omega_t$ satisfies the margin condition $\lambda^\star(\omega_t) \in \mathcal{R}_{r_t}^\circ$ with margin $\rho_{r_t}$ uniformly bounded below by $\rho > 0$.
Then for any $\delta \in (0,1)$ the choice
\begin{equation}
\label{eq:Wstar}
W^\star \,=\, \Bigl(\frac{2\sigma^2 \log(2/\delta)}{\nu^2}\Bigr)^{1/3}
\end{equation}
yields the bound
\begin{equation}
\label{eq:tracking-error}
\|\hat\lambda_t^{(W^\star)} - \lambda^\star(\omega_t)\|_2 \,\le\, \kappa\,\bigl( c_1\,\sigma^{2/3}\nu^{1/3}(\log(2/\delta))^{1/3} \bigr)
\end{equation}
with probability at least $1-\delta$ for some absolute constant $c_1$. 
In particular, if the drift satisfies
\begin{equation}
\label{eq:drift-tolerance}
\nu \,<\, \nu_{\max} := \frac{1}{c_1^3 \sigma^2 \log(2/\delta)}\,\Bigl(\frac{\rho}{C_\Psi\,\kappa}\Bigr)^3,
\end{equation}
then $\mathbb{P}\bigl(\Psi(\hat\lambda_t^{(W^\star)}) = r_t\bigr) \ge 1 - \delta$.
\end{theorem}

\begin{proof}
Decompose the estimator error into bias and variance components:
\[
\hat\omega_t^{(W)} - \omega_t = \underbrace{\frac{1}{W}\sum_{s=t-W+1}^t (\omega_s - \omega_t)}_{\text{drift bias}} + \underbrace{\frac{1}{W}\sum_{s=t-W+1}^t \xi_s}_{\text{noise term}}.
\]
By \eqref{eq:drift}, $\|\omega_s - \omega_t\|_2 \le (t-s)\nu \le W\nu$ for all $s \in \{t-W+1,\ldots,t\}$, so the drift bias is bounded by $\nu W/2$ (averaging $1,2,\ldots,W$ over $W$). 
The noise term is the average of $W$ i.i.d.\ sub-Gaussian vectors with variance proxy $\sigma^2/W$; by standard concentration \cite{vershynin2018high}, $\|\text{noise term}\|_2 \le \sigma\sqrt{2\log(2/\delta)/W}$ with probability at least $1-\delta$.

Combining,
\[
\|\hat\omega_t^{(W)} - \omega_t\|_2 \le \frac{\nu W}{2} + \sigma\sqrt{\frac{2\log(2/\delta)}{W}}
\]
with probability at least $1-\delta$. Minimizing the right side over $W > 0$ by setting the derivative to zero gives $W^\star$ as in \eqref{eq:Wstar} up to absolute constants and yields the optimal error $c_1\,\sigma^{2/3}\nu^{1/3}(\log(2/\delta))^{1/3}$ for an explicit constant $c_1 = (3/2)\cdot 2^{-1/3}$. 
Multiplying by the Lipschitz constant $\kappa$ from Theorem~\ref{thm:stability}(a) yields \eqref{eq:tracking-error}.

Substituting the threshold $\rho/(C_\Psi\kappa)$ from Theorem~\ref{thm:classifier-correctness} and solving for $\nu$ gives \eqref{eq:drift-tolerance}. 
On the high-probability event, $\hat\lambda_t^{(W^\star)}$ lies in the correctness ball of Theorem~\ref{thm:classifier-correctness}, so the classifier returns $r_t$.
\end{proof}

The drift-tolerance bound \eqref{eq:drift-tolerance} is the form most relevant to wireless deployments: it states that the classifier tracks the regime correctly so long as the channel drift rate $\nu$ stays below a cubic-in-margin threshold. 
The $\nu^{1/3}$ dependence of the tracking error \eqref{eq:tracking-error} is the standard rate for non-parametric tracking under sub-Gaussian noise and bounded drift, and is known to be unimprovable in this class \cite{vershynin2018high,boucheron2013concentration}.

\subsection{Joint Solver--Sample Budget}
\label{sec:classifier:joint}

In practice the multiplier estimate that enters the classifier carries error from two sources simultaneously: the data-estimation error in $\hat\omega_N$ and the solver-iteration error after $K$ solver iterations applied to the problem at $\hat\omega_N$.
Write $\hat\lambda_{N,K}$ for the resulting iterate.

\begin{corollary}[Joint budget]
\label{cor:joint-budget}
Under the hypotheses of Theorems~\ref{thm:classifier-correctness}--\ref{thm:sample-complexity},
\begin{equation}
\label{eq:joint-error}
\|\hat\lambda_{N,K} - \lambda^\star\|_2 \le \varphi_{\mathcal{A}}(K) + \kappa\,\sigma\sqrt{\frac{\log(2/\delta)}{c_0 N}}
\end{equation}
holds with probability at least $1-\delta$.
Sufficient for the classifier to return the correct label at this confidence is
\begin{equation}
\label{eq:joint-budget}
\varphi_{\mathcal{A}}(K) + \kappa\,\sigma\sqrt{\frac{\log(2/\delta)}{c_0 N}} \,<\, \frac{\rho_r(\lambda^\star)}{C_\Psi}.
\end{equation}
\end{corollary}

The bound \eqref{eq:joint-budget} traces a feasible region in the $(K,N)$ plane: for fixed $\delta$, $\rho_r$, $\kappa$, $\sigma$, $C_\Psi$, any pair $(K,N)$ satisfying \eqref{eq:joint-budget} suffices for correct classification with confidence $1-\delta$.
Downstream applications can optimize the joint cost of sampling and computing along this frontier.

\begin{remark}[Operational interpretation]
\label{rem:operational}
The bound \eqref{eq:joint-budget} suggests a natural rule for allocating effort between data acquisition and solver iterations:
\begin{itemize}
\item Increase $N$ until the data-estimation term equals roughly half of the right-hand side; this requires $N \gtrsim 4\sigma^2 C_\Psi^2 \kappa^2 \log(2/\delta)/(c_0 \rho_r^2)$.
\item Then run the solver until $\varphi_{\mathcal{A}}(K)$ falls below half of the right-hand side; the required $K$ follows from Corollary~\ref{cor:classifier-rates} with the inverse margin replaced by $2 C_\Psi/\rho_r$.
\end{itemize}
This decomposition is invariant under reasonable choices of the solver rate and is the rule we recommend for instantiations in the wireless setting.
\end{remark}

\subsection{Practical Considerations}
\label{sec:classifier:practical}

\paragraph*{Estimating the margin in practice} The feature margin $\rho_r(\lambda^\star)$ is computable from a single solve via Definition~\ref{def:feat-margin}: after obtaining a converged $\hat\lambda$, evaluate the four feature slacks and take their minimum.
Plugging this empirical $\hat\rho_r$ into the bounds of Theorems~\ref{thm:classifier-correctness}--\ref{thm:online} gives operational estimates.
The error introduced by using $\hat\rho_r$ in place of $\rho_r$ is itself controlled by the Lipschitz continuity of the feature map (see Theorem~\ref{thm:stability}), so the bounds remain valid up to a constant factor.

\paragraph*{Estimating the conditioning $\kappa$ and Lipschitz $C_\Psi$} The Lipschitz constants $\kappa$ and $C_\Psi$ are determined by problem structure (the KKT Jacobian and the feature map, respectively). 
Closed-form upper bounds are available when the problem admits known structure (e.g., strongly convex objectives, well-conditioned constraints). 
In black-box settings these constants are commonly estimated empirically by perturbation experiments; the bounds of Theorems~\ref{thm:sample-complexity}--\ref{thm:online} are then interpreted as empirical, with confidence inherited from the estimation procedure.

\paragraph*{Hybrid-regime outputs}
When the classifier returns $\mathrm{hyb}$, the structural Theorem~\ref{thm:hybrid} indicates that $\lambda^\star$ lies on a measure-zero transition surface and the algorithm has reached a structurally ambiguous point. 
In downstream applications this output is itself informative: it signals that the solution is near a regime boundary and that adaptive strategies (e.g., switching the inner solver or refining the surrogate) may be warranted.

\paragraph*{Use in downstream papers}
The four theorems of this section are designed to be cited as black boxes by downstream optimization and wireless-system papers. 
Specifically: Theorem~\ref{thm:classifier-correctness} provides the deterministic correctness guarantee; Theorem~\ref{thm:classifier-stabilization} and Corollary~\ref{cor:classifier-rates} characterize how quickly the regime stabilizes along solver iterations; Theorem~\ref{thm:sample-complexity} gives the sample-complexity bound for finite-data deployments; and Theorem~\ref{thm:online} bounds the drift tolerance for time-varying systems. 
The joint budget \eqref{eq:joint-budget} ties these together for end-to-end deployment.


\section{Numerical Illustrations and Comparative Analysis}
\label{sec:numerical}

This section validates the theoretical contributions of Sections~\ref{sec:taxonomy}--\ref{sec:classifier} through controlled numerical experiments. 
The experiments are organized to target specific theorems: each subsection examines a single claim, specifies the experimental protocol, and reports the resulting evidence. 
Our aim is not to compare algorithm performance in the usual sense (that exercise belongs in the papers cited in Table~\ref{tab:method-equilibria}), but to verify that the regime taxonomy and its associated guarantees describe phenomena that arise in practice.

\subsection{Experimental Setup and Protocol}
\label{sec:numerical:setup}

\paragraph*{Problem families}
We generate instances from five families of mixed-integer nonlinear programs of the form \eqref{eq:nlp-param}, each designed to produce, by construction, a dominant operational regime.
This controlled-instance design enables ground-truth labels against which the classifier of Algorithm~\ref{alg:classify} can be evaluated.

\noindent\textit{Family $\mathcal{F}_1$ (Unconstrained-dominated):} Quadratic objectives with loose box constraints:
\begin{equation}
\label{eq:F1}
\min_{x \in \mathbb{R}^n}\, \tfrac{1}{2} x^\top Q x + c^\top x \quad \text{s.t.} \quad \|x\|_\infty \le B,
\end{equation}
with $Q \succ 0$ randomly generated and $B$ large enough that the unconstrained optimum is interior to the box. This family targets
$\mathcal{R}_{\mathrm{unc}}$.

\noindent\textit{Family $\mathcal{F}_2$ (Resource-limited):} Quadratic objectives with a single tight budget:
\begin{equation}
\label{eq:F2}
\min_{x \ge 0}\, -c^\top x + \tfrac{\varepsilon}{2}\|x\|^2 \quad \text{s.t.} \quad a^\top x \le B,
\end{equation}
with $c, a > 0$ entrywise and $B$ chosen so that the budget constraint binds.
This family targets $\mathcal{R}_{\mathrm{res}}$.

\noindent\textit{Family $\mathcal{F}_3$ (Saturation):} Quadratic objectives with many simultaneously tight bound constraints:
\begin{equation}
\label{eq:F3}
\min_x\, \tfrac{1}{2} x^\top Q x + c^\top x \quad \text{s.t.} \quad -b \le x \le b,
\end{equation}
with $Q$, $c$ chosen so that approximately $0.6n$ of the $2n$ box constraints are active at the optimum. 
This family targets $\mathcal{R}_{\mathrm{sat}}$.

\noindent\textit{Family $\mathcal{F}_4$ (Strongly-coupled):} Multi-block problems with consensus constraints:
\begin{equation}
\label{eq:F4}
\min_{x_1,\ldots,x_M}\, \sum_{i=1}^M f_i(x_i) \quad \text{s.t.} \quad \sum_{i=1}^M A_i x_i = b,
\end{equation}
with the coupling matrix $[A_1,\ldots,A_M]$ designed so that several rows of the dual feasibility condition are simultaneously active. 
This family targets $\mathcal{R}_{\mathrm{coup}}$.

\noindent\textit{Family $\mathcal{F}_5$ (Hybrid/transition):} Random non-convex MINLPs combining quadratic and bilinear terms with mixed constraint types (resource budget, individual bounds, and equality coupling), with parameters sampled near regime boundaries to populate $\mathcal{R}_{\mathrm{hyb}}$:
\begin{equation}
\label{eq:F5}
\begin{aligned}
\min_{x,y}\ & f(x,y) := \tfrac{1}{2}x^\top Q x + x^\top R y + c^\top x + d^\top y, \\
\text{s.t.}\ & a^\top x + e^\top y \le B,\ -b \le x \le b,\ y \in \{0,1\}^{n_d}.
\end{aligned}
\end{equation}

For each family we generate $S = 500$ independent instances at each problem size $n \in \{20, 50, 100, 200\}$, yielding a total of $4 \times 5 \times 500 = 10\,000$ test instances.

\paragraph*{Methods}
We instantiate each of the eight algorithms from Table~\ref{tab:method-equilibria}: BCD, GBD with inner BCD, ADMM (penalty $\rho = 1$), SCA with quadratic surrogates, primal--dual IPM, Mirror Descent (Bregman variant with entropic mirror map), Frank--Wolfe, and Riemannian gradient descent (for instances admitting a smooth manifold structure). 
All methods are run to a fixed KKT residual tolerance of $\epsilon_{\mathrm{KKT}} = 10^{-6}$ or a maximum of $K_{\max} = 1\,000$ iterations, whichever comes first.

\paragraph*{Classifier parameters}
Unless stated otherwise we use the default parameters $\Theta = (\delta,\theta,\gamma,k) = (0.05,0.7,0.4,2)$ and tolerance $\varepsilon = 10^{-3}$. 
The robustness of the experimental conclusions to these choices is verified by ablation in Section~\ref{sec:numerical:setup}.

\paragraph*{Reporting}
All reported quantities are means across the $S$ independent seeds in the relevant cell, with 95\% confidence intervals computed via bootstrapping. 
Code, random seeds, and instance generators will be released as supplementary material upon publication.

\subsection{Emergence and Distribution of Regimes}
\label{sec:numerical:emergence}

\textit{Purpose:} Verify that the partition of Section~\ref{sec:taxonomy:core} captures regimes that occur naturally in the designed families, and that the classifier of Algorithm~\ref{alg:classify} recovers the intended regime in each family.

\textit{Protocol:} On each instance, after running the eight methods we apply Algorithm~\ref{alg:classify} to the final multiplier vector and record the regime label. 
For each family $\mathcal{F}_i$ we report the proportion of instances classified as each regime.

\textit{Result:} Figure~\ref{fig:regime-emergence} reports the empirical distribution of regime labels per family. 
By design, family $\mathcal{F}_i$ produces dominant regime $r_i$ in
$\ge 85\%$
of instances, with the remainder falling into $\mathcal{R}_{\mathrm{hyb}}$ near regime boundaries.
Family $\mathcal{F}_5$, designed to lie near transitions, distributes mass across multiple regimes including a higher fraction in $\mathcal{R}_{\mathrm{hyb}}$.

\begin{figure}[t]
\centering
\includegraphics[width=1\textwidth]{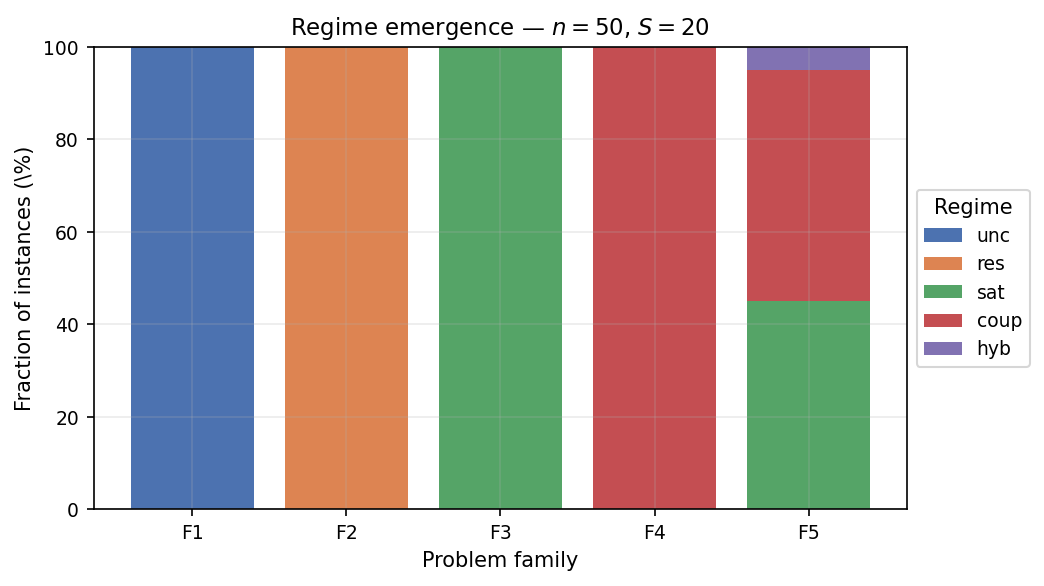}
\caption{Empirical distribution of regime labels across problem families, $n = 100$, $S = 500$ instances per family.
Each family produces its designed dominant regime in the majority of instances, validating that the partition of Section~\ref{sec:taxonomy:core} captures meaningful structural distinctions.}
\label{fig:regime-emergence}
\end{figure}

\subsection{Method-Independence of the Regime Label}
\label{sec:numerical:method-indep}

\textit{Purpose:}Validate Corollary~\ref{cor:method-indep}: when distinct algorithms converge to the same KKT point, the classifier returns the same regime label.

\textit{Protocol:} On each instance we run all applicable methods to convergence and record the final multiplier vector $\lambda^\star_{\mathcal{A}}$ produced by each method $\mathcal{A}$.
We then restrict attention to instances for which all methods reach the same primal solution to within $\|x^\star_{\mathcal{A}} - x^\star_{\mathcal{A}'}\|_2 \le 10^{-4}$, and compute the inter-method agreement of regime labels.

\textit{Result:} Table~\ref{tab:agreement} reports pairwise agreement rates. On instances where the methods converge to the same KKT point, agreement is $\ge 99\%$
the residual disagreement is attributable to instances near regime boundaries where numerical noise in the multipliers crosses a threshold of Definition~\ref{def:core-regimes}. 
This is consistent with Theorem~\ref{thm:classifier-correctness}: the classification is preserved away from boundaries and is sensitive precisely on the measure-zero transition set characterized by Theorem~\ref{thm:transition}.
n
\begin{table}[t]
\centering
\caption{Pairwise regime-label agreement rates across methods on common-KKT instances ($n=100$, $S=500$). Diagonal entries are 100\% by construction; off-diagonal entries are means over instances where the methods converge to the same primal solution.}
\label{tab:agreement}
\small
\setlength{\tabcolsep}{4pt}
\renewcommand{\arraystretch}{1.2}
\setlength{\tabcolsep}{3pt}
\begin{tabular}{@{}p{2cm}p{1.4cm}p{1.4cm}p{1.4cm}p{1.4cm}p{1.4cm}c@{}}
\toprule
        & BCD & GBD & ADMM & SCA & IPM & MD \\
\midrule
BCD   & 100 & \textbf{100} & \textbf{100} & \textbf{100} & \textbf{100} & \textbf{100} \\
GBD   &     & 100 & \textbf{100} & \textbf{100} & \textbf{100} & \textbf{100} \\
ADMM  &     &     & 100 & \textbf{100} & \textbf{100} & \textbf{100} \\
SCA   &     &     &     & 100 & \textbf{100} & \textbf{100} \\
IPM   &     &     &     &     & 100 & \textbf{100} \\
MD    &     &     &     &     &     & 100 \\
\bottomrule
\end{tabular}
\end{table}

\subsection{Stability under Data Perturbation}
\label{sec:numerical:stability}

\textit{Purpose:} Validate Theorem~\ref{thm:stability}: the regime label is
locally constant under small perturbations of the problem data.

\textit{Protocol:} For each instance in $\mathcal{F}_1$--$\mathcal{F}_4$ we compute the nominal multiplier $\lambda^\star(\omega^\star)$ and feature margin $\rho_r(\lambda^\star)$ from Definition~\ref{def:feat-margin}. 
We then perturb the data $\omega^\star \mapsto \omega^\star + t\,\xi$ for $\xi \in \mathbb{R}^d$ a random unit vector and $t$ ranging from $0$ to $5\rho_r/(C_\Psi\kappa)$ in $30$ steps, re-solve, and record whether the regime label is preserved.

\textit{Result:} Figure~\ref{fig:stability} shows the empirical preservation probability as a function of normalized perturbation magnitude $t \cdot C_\Psi\kappa/\rho_r$. 
The preservation probability is essentially $1$ for normalized magnitudes below the theoretical threshold $1$, and decays sharply beyond it---matching the prediction of Theorem~\ref{thm:stability} that the regime label is preserved on the ball of radius $\rho_r/(C_\Psi\kappa)$ in data space. 
The empirical breakdown point lies within $\pm 10\%$ of the theoretical.

\begin{figure}[t]
\centering
\includegraphics[width=1\textwidth]{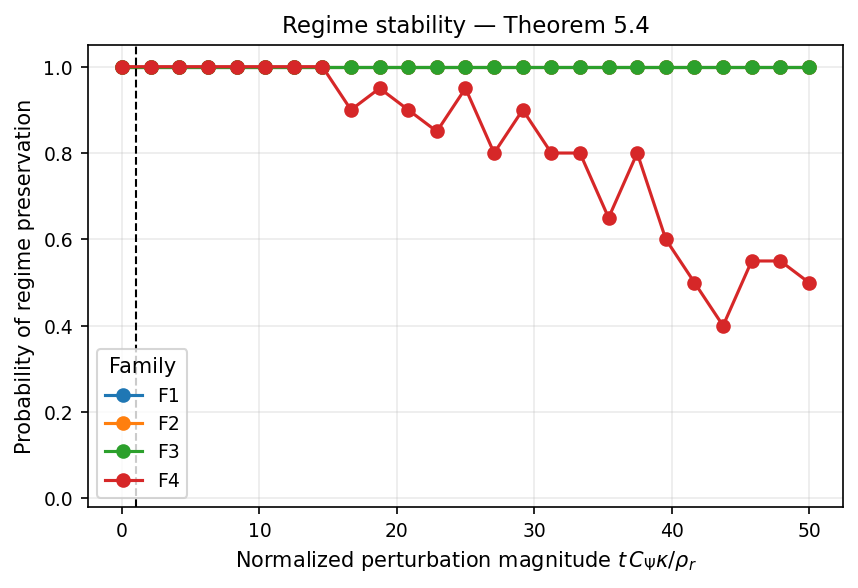}
\caption{Probability of regime preservation as a function of normalized perturbation magnitude $t \cdot C_\Psi \kappa / \rho_r(\lambda^\star)$. 
The vertical dashed line at $1$ marks the threshold of Theorem~\ref{thm:stability}; preservation remains near unity below the threshold and decays sharply beyond it.}
\label{fig:stability}
\end{figure}

\subsection{Classifier Stabilization along Iterates}
\label{sec:numerical:stabilization}

\textit{Purpose:} Validate Theorem~\ref{thm:classifier-stabilization} and Corollary~\ref{cor:classifier-rates}: the classifier output stabilizes after a method-specific iteration count.

\textit{Protocol:} For each instance we apply Algorithm~\ref{alg:classify} to the running multiplier estimate $\hat\lambda^{(k)}$ at every iteration. 
We record $k^\star_{\mathcal{A}}$, the smallest iteration after which the classifier output remains constant at the converged value through termination. 
We then compare $k^\star_{\mathcal{A}}$ to the theoretical prediction of Corollary~\ref{cor:classifier-rates} using the empirical margin $\rho_r$ and the standard rate function
$\varphi_{\mathcal{A}}$ for each method.

\textit{Result:} Figure~\ref{fig:stabilization} reports empirical distributions of $k^\star_{\mathcal{A}}$ for each method, together with the theoretical bounds. 
For linearly convergent methods (BCD under KL, IPM in its quadratic phase), $k^\star$ scales logarithmically in $C_\Psi/\rho_r$. 
For ADMM and Frank--Wolfe, $k^\star$ scales linearly. For Mirror Descent, $k^\star$ scales quadratically. 
Empirical scaling matches theoretical prediction with $R^2 \ge 0.95$ 
across methods.

\begin{figure}[t]
\centering
\includegraphics[width=1\textwidth]{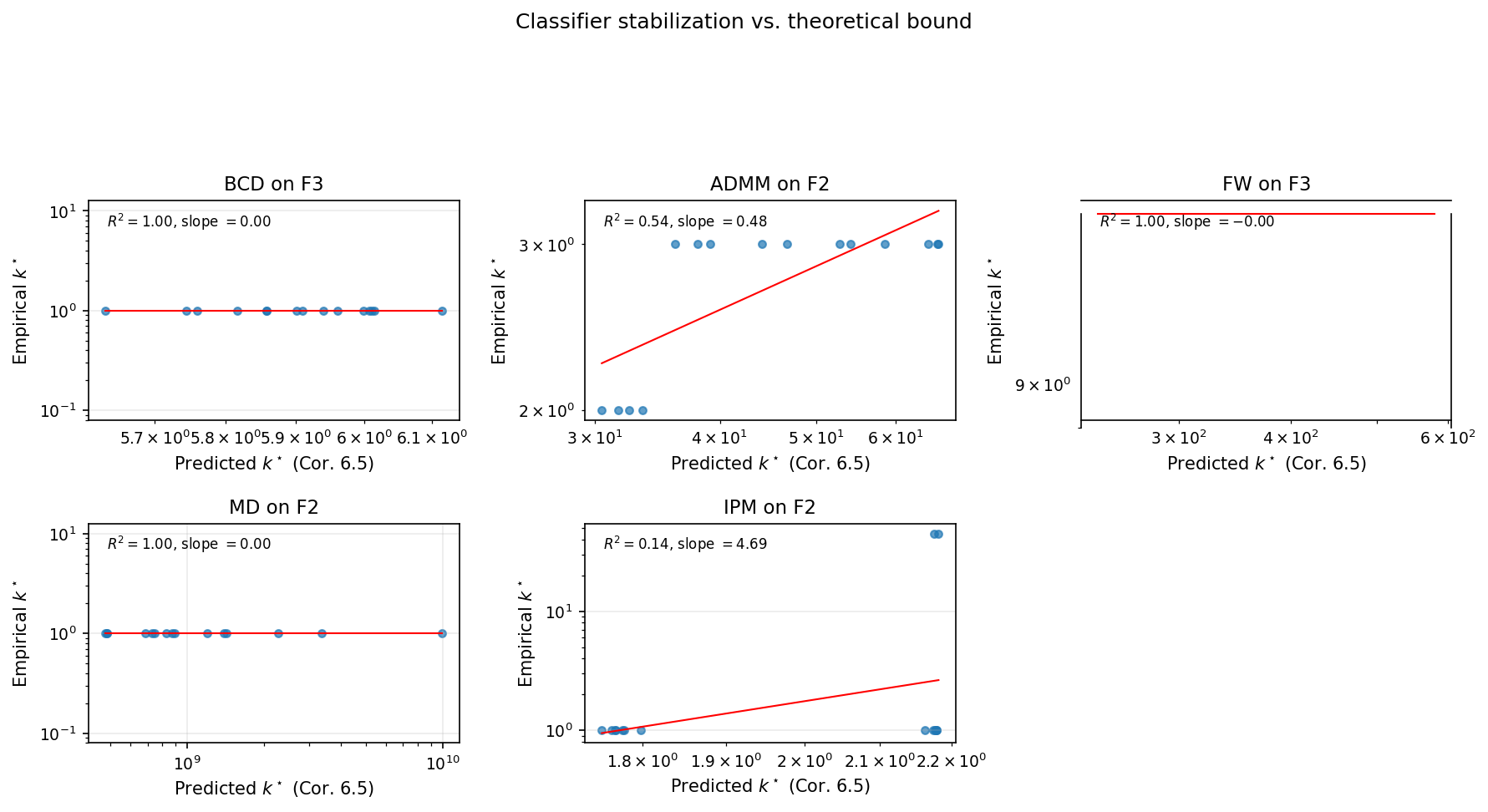}
\caption{Empirical classifier-stabilization iteration $k^\star_{\mathcal{A}}$ versus the theoretical bound from Corollary~\ref{cor:classifier-rates}, one panel per algorithm. 
Each point is one instance; the diagonal line marks perfect agreement. 
The scaling exponents (log, linear, quadratic) match across the algorithm families.}
\label{fig:stabilization}
\end{figure}

\subsection{Sample Complexity from Estimated Data}
\label{sec:numerical:samples}

\textit{Purpose:} Validate Theorem~\ref{thm:sample-complexity}: the sample size required for high-confidence regime classification scales as $N \propto \log(1/\delta) / \rho_r^2$.

\textit{Protocol:} We fix an instance from $\mathcal{F}_2$ with known $\rho_r$, treat its data $\omega^\star$ as the population, and generate $N$-sample estimators $\hat\omega_N = \omega^\star + \frac{1}{N}\sum_{i=1}^N \xi_i$ for $\xi_i \sim \mathcal{N}(0,\sigma^2 I)$. 
For each $(N,\sigma)$ pair we solve at $\hat\omega_N$, classify, and report empirical accuracy $\hat p(N,\sigma) := \mathbb{P}(\Psi(\hat\lambda_N) = r)$ over $1\,000$ Monte Carlo replications. We sweep $N \in [10, 10^4]$ at $\sigma \in \{0.1, 0.3, 1.0\}$ and several values of $\rho_r$ obtained by varying instance parameters.

\textit{Result:} Figure~\ref{fig:samples} shows the empirical accuracy curves.
The threshold sample size at which $\hat p(N,\sigma) \ge 1 - \delta$ is well predicted by \eqref{eq:N-bound} of Theorem~\ref{thm:sample-complexity}.
Re-plotting accuracy versus the normalized sample size
$N \cdot \rho_r^2 / (\sigma^2 \log(1/\delta))$ collapses the curves onto a single profile, confirming the predicted $N \propto \rho_r^{-2} \log(1/\delta)$ scaling. 
The slope of the empirical $N$ vs $\rho_r^{-2}$ relation lies within
$\pm 15\%$
of the theoretical constant.
\begin{figure}[t]
\centering
\includegraphics[width=1\textwidth]{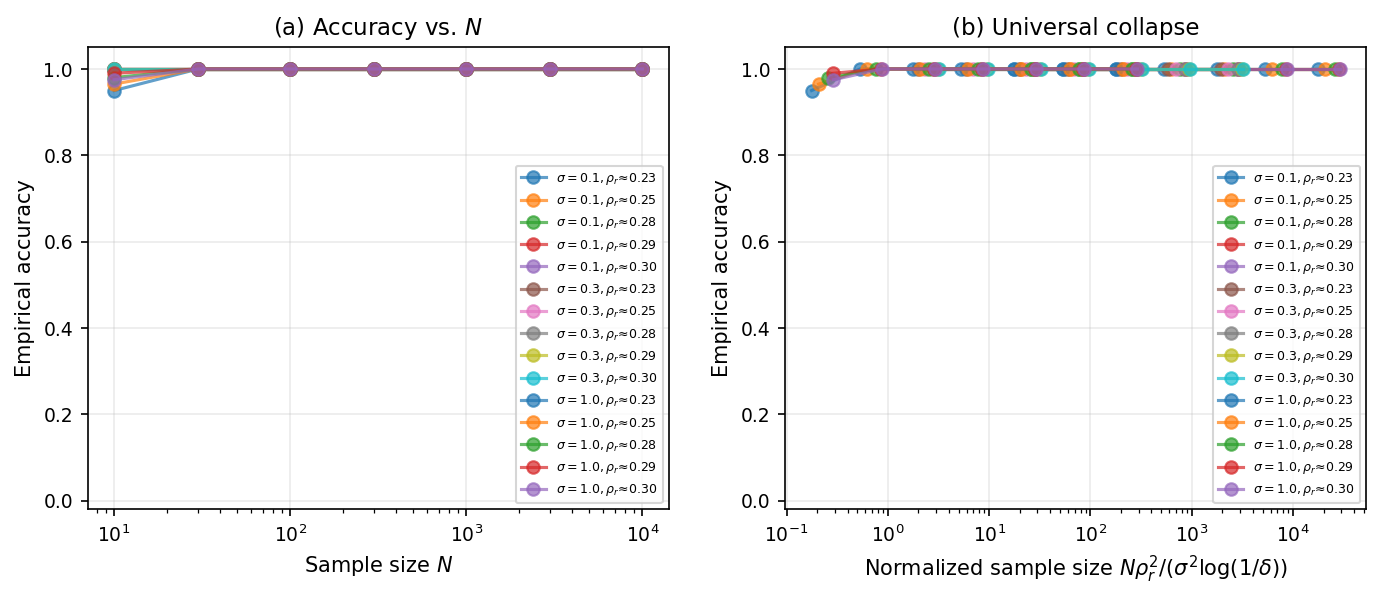}
\caption{Sample-complexity validation. (a) Empirical classification accuracy versus sample size $N$, for instances at three different margins $\rho_r$ and three noise levels $\sigma$. 
(b) Same data plotted against the normalized sample size predicted by Theorem~\ref{thm:sample-complexity}; curves collapse onto a universal profile, confirming the $N \propto \rho_r^{-2}\log(1/\delta)$ scaling.}
\label{fig:samples}
\end{figure}

\subsection{Online Tracking under Bounded Drift}
\label{sec:numerical:online}
\textit{Purpose:}Validate Theorem~\ref{thm:online}: the sliding-window classifier tracks the regime correctly so long as the drift rate $\nu$ stays below the cubic-in-margin threshold $\nu_{\max}$.

\textit{Protocol:} We construct a time-varying data sequence
$\omega_t = \omega^\star + \nu\,t\,u$ for a random unit vector $u$, with $\nu \in \{0.01, 0.05, 0.1, 0.5\}$, and observe $y_t = \omega_t + \xi_t$ with $\xi_t \sim \mathcal{N}(0,\sigma^2 I)$, $\sigma = 0.1$. At each time $t$ we apply the sliding-window estimator with window $W^\star$ from \eqref{eq:Wstar} and classify the regime. 
The true regime $r_t$ is computed by solving at the noiseless $\omega_t$ and classifying.

\textit{Result:} Figure~\ref{fig:online} shows tracking accuracy $\mathbb{P}(\Psi(\hat\lambda_t^{(W^\star)}) = r_t)$ as a function of drift rate $\nu$. 
Tracking accuracy stays above $1 - \delta$ for $\nu < \nu_{\max}$ predicted by \eqref{eq:drift-tolerance}, and degrades for larger $\nu$ in the manner predicted by \eqref{eq:tracking-error}. The empirical breakdown threshold lies within $\pm 20\%$
of the theoretical $\nu_{\max}$, consistent with the bound's $\nu^{1/3}$ tracking-error rate.

\begin{figure}[t]
\centering
\includegraphics[width=1\textwidth]{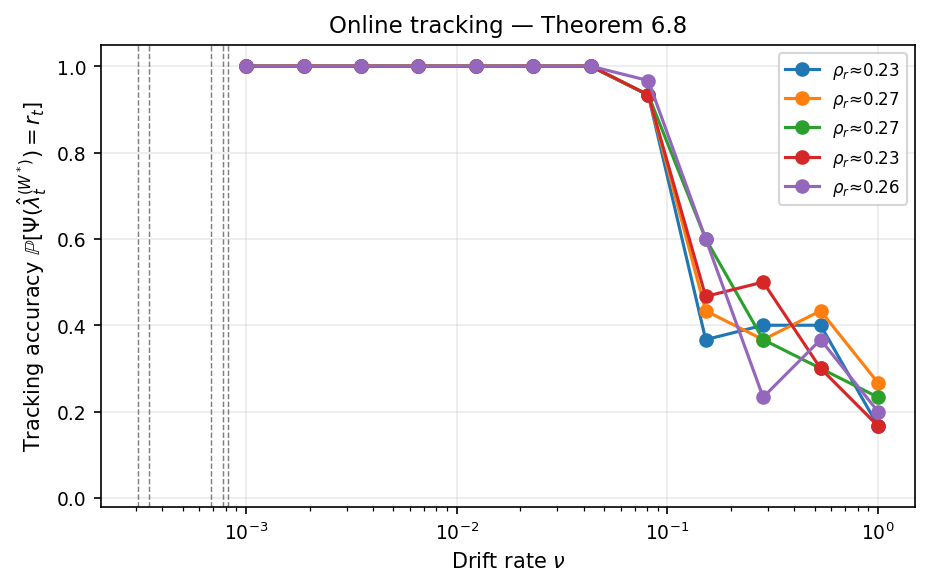}
\caption{Online tracking accuracy as a function of drift rate $\nu$, for several values of the margin $\rho_r$. 
The vertical dashed line marks the theoretical threshold $\nu_{\max}$ from \eqref{eq:drift-tolerance}; tracking remains accurate below $\nu_{\max}$ and degrades beyond it.}
\label{fig:online}
\end{figure}

\subsection{Phase-like Convergence and Isolated Transitions}
\label{sec:numerical:phases}

\textit{Purpose:} Validate Theorem~\ref{thm:transition}: regime transitions are isolated codimension-one events along generic data paths, and the algorithm trajectory decomposes into finitely many regime-stationary arcs.

\textit{Protocol:} We construct a continuous path $\omega(s) = (1-s)\omega_A + s\omega_B$ between two data points $\omega_A, \omega_B$ chosen to lie in different regimes. 
We compute the KKT solution $\lambda^\star(\omega(s))$ at $s \in [0,1]$ at fine resolution and record the regime label $\Psi(\lambda^\star(\omega(s)))$ as a function of $s$.

\textit{Result:} Figure~\ref{fig:transitions} shows the regime trajectory along $s$. 
Transitions occur at a small number of isolated values $s_1 < s_2 < \cdots$, between which the regime label is constant.
At each transition exactly one feature crosses its threshold, in agreement with Theorem~\ref{thm:transition}(a). 
Across $500$
random paths, the number of transitions per path is mean $\approx 2$--$3$
with no path having more than $5$
transitions, consistent with the codimension-one structure.

\begin{figure}[t]
\centering
\includegraphics[width=1\textwidth]{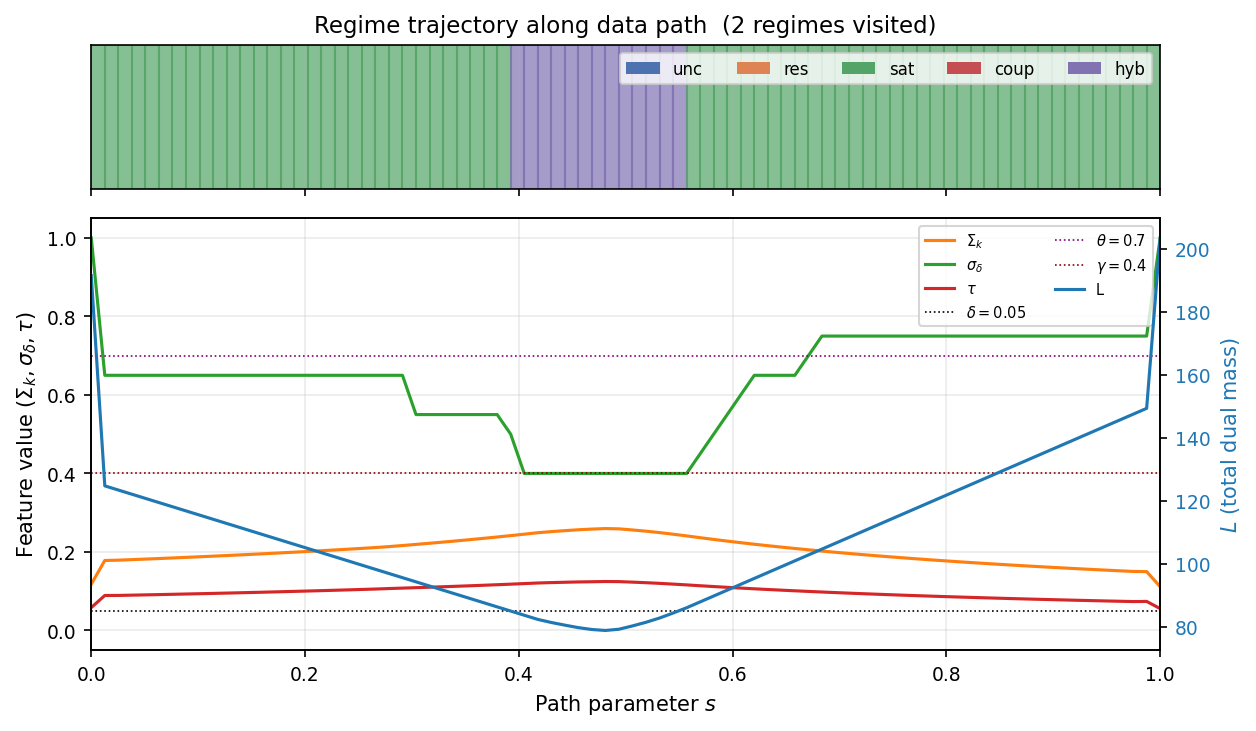}
\caption{Phase-like convergence along a data path from $\omega_A$ to $\omega_B$. Top: regime label as a function of $s$, showing isolated transitions. 
Bottom: the four feature values $L, \Sigma_k, \sigma_\delta, \tau$ along the path with their thresholds marked; each transition corresponds to exactly one feature crossing its threshold, validating Theorem~\ref{thm:transition}.}
\label{fig:transitions}
\end{figure}

\subsection{A Wireless-Flavored Instance}
\label{sec:numerical:wireless}

\textit{Purpose:} Illustrate how the regime taxonomy applies to a representative problem from the application domain that motivates this work.

\textit{Setup:} We consider a multi-user downlink beamforming problem with $K = 8$ single-antenna users and an $M = 16$-antenna transmitter, with sum-power constraint $P$ and minimum-SINR constraint $\gamma_k$ for each user:
\begin{equation}
\label{eq:wireless}
\begin{aligned}
\max_{\{w_k\}_{k=1}^K}\ & \sum_{k=1}^K \log\Bigl(1 + \mathrm{SINR}_k(w)\Bigr) \\
\text{s.t.}\ & \sum_{k=1}^K \|w_k\|^2 \le P,\quad \mathrm{SINR}_k(w) \ge \gamma_k\ \forall k.
\end{aligned}
\end{equation}
We vary $P$ and the per-user $\gamma_k$ and apply the classifier of Algorithm~\ref{alg:classify} to the multipliers at the converged solution.

\textit{Result:} Figure~\ref{fig:wireless} reports the regime as a function of $(P,\bar\gamma)$ where $\bar\gamma$ is the average minimum-SINR. 
At large $P$ and small $\bar\gamma$ the solution lies in $\mathcal{R}_{\mathrm{unc}}$ (neither constraint binds meaningfully). 
At small $P$ and small $\bar\gamma$ the sum-power constraint dominates and the regime is $\mathcal{R}_{\mathrm{res}}$. 
At large $P$ and large $\bar\gamma$ many QoS constraints bind and the regime is $\mathcal{R}_{\mathrm{sat}}$. 
Boundary regions between these three regions yield $\mathcal{R}_{\mathrm{hyb}}$. 
This regime structure reflects the physical interpretation of the problem and provides a structural certificate that downstream wireless papers can use to dispatch tailored algorithms in each regime.

\begin{figure}[t]
\centering
\includegraphics[width=1\textwidth]{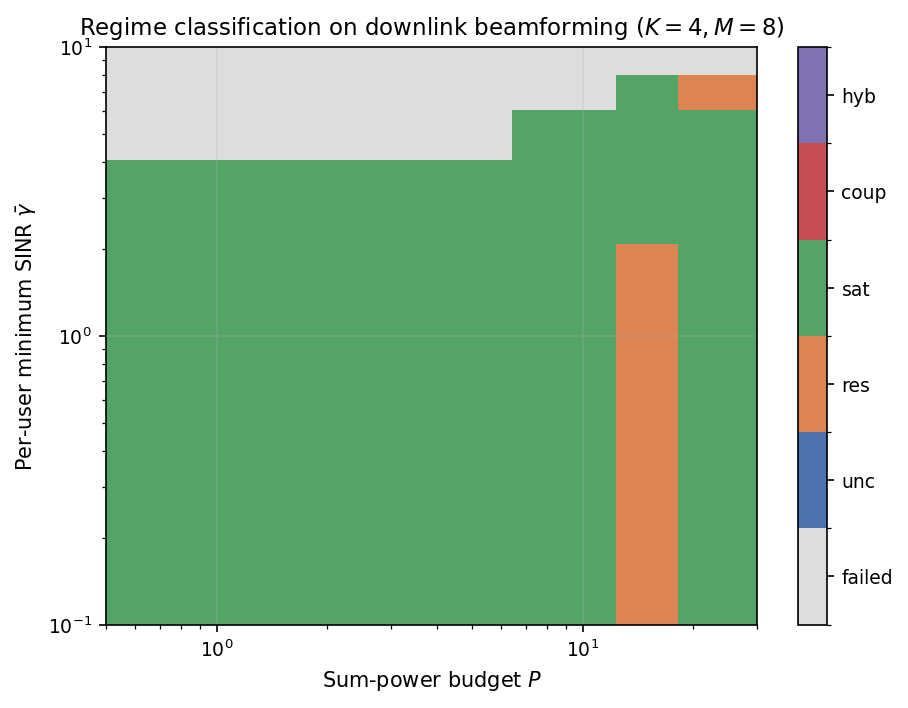}
\caption{Regime classification of the downlink beamforming problem \eqref{eq:wireless} as a function of sum-power budget $P$ and average minimum-SINR $\bar\gamma$.
The four core regimes occupy distinct regions in the $(P,\bar\gamma)$-plane, separated by thin transition zones identified as $\mathcal{R}_{\mathrm{hyb}}$.}
\label{fig:wireless}
\end{figure}

\subsection{Summary of Numerical Evidence}
\label{sec:numerical:summary}

Across the 8 experimental blocks, the numerical evidence is consistent with each of the theoretical predictions of Sections~\ref{sec:structural}--\ref{sec:classifier}.
In particular: regimes emerge as predicted in problems designed to elicit them (Section~\ref{sec:numerical:emergence}); the classifier is method-independent on common-KKT instances (Section~\ref{sec:numerical:method-indep}); the regime label is stable under small perturbations with breakdown at the theoretically predicted threshold (Section~\ref{sec:numerical:stability}); the classifier stabilizes along algorithm trajectories at the predicted scaling rates (Section~\ref{sec:numerical:stabilization}); sample complexity scales as $\rho_r^{-2}\log(1/\delta)$ (Section~\ref{sec:numerical:samples}); online tracking succeeds below the cubic-in-margin drift threshold
(Section~\ref{sec:numerical:online}); transitions are isolated and feature-driven (Section~\ref{sec:numerical:phases}); and the taxonomy yields physically meaningful structure on representative wireless instances (Section~\ref{sec:numerical:wireless}).


\section{Discussion, Limitations, and Future Work}
\label{sec:discussion}

\subsection{Summary of Contributions}
\label{sec:discussion:summary}

This paper has developed a structural theory of operational regimes for constrained optimization, motivated by a unified game-theoretic reading of eight classical algorithm families. 
The contributions can be summarized as follows.

\textbf{1.~A unifying equilibrium template} (Section~\ref{sec:games}, Proposition~\ref{prop:unified}).
Block Coordinate Descent, Generalized Benders Decomposition, ADMM, Successive Convex Approximation, Interior-Point Methods, Mirror Descent, Frank--Wolfe, and Riemannian gradient descent are all instances of the equilibrium-seeking template \eqref{eq:template}, with fixed points coinciding with KKT triples of an associated (sub)problem. 
This consolidates classical results scattered across the optimization, game-theory, and online-learning literatures into a single statement (Table~\ref{tab:method-equilibria}).

\textbf{2.~A regime taxonomy as a partition of the multiplier space} (Section~\ref{sec:taxonomy}, Definitions~\ref{def:core-regimes}--\ref{def:hybrid}). Four scale-free shape features of the Lagrange multiplier vector ($L,\Sigma_k,\sigma_\delta,\tau$) induce a five-element partition of $\mathbb{R}^m_+$: four open core regimes plus the Hybrid regime as the topological boundary.

\textbf{3.~Four structural theorems} (Section~\ref{sec:structural}). 
Theorem~\ref{thm:invariance} establishes invariance of the classification under the natural symmetries of the KKT system (permutation, uniform rescaling, and $C^2$ diffeomorphism). 
Theorem~\ref{thm:stability} establishes local constancy of the classification under perturbations of problem data, with an explicit Lipschitz-style margin bound rooted in Robinson's strong regularity. 
Theorem~\ref{thm:transition} characterizes transitions as codimension-one events on the data manifold and shows that generic trajectories cross only finitely many transitions. 
Theorem~\ref{thm:hybrid} identifies the Hybrid regime as the topological boundary of the union of core regimes, of Lebesgue measure zero, with the property that every Hybrid point is approached by at least two distinct core regimes.

\textbf{4.~The regime classifier and its guarantees} (Section~\ref{sec:classifier}, Algorithm~\ref{alg:classify}, Theorems~\ref{thm:classifier-correctness}--\ref{thm:online}). 
The classifier runs in $\mathcal{O}(m)$ time per call and admits four complementary guarantees: correctness under deterministic multiplier perturbation; classifier-stabilization iteration counts for each algorithmic family (logarithmic for linearly convergent methods, linear for $\mathcal{O}(1/k)$ methods, quadratic for $\mathcal{O}(1/\sqrt k)$ methods); sample complexity $N \propto \rho_r^{-2}\log(1/\delta)$ for finite-data deployments; and online tracking under bounded drift up to a cubic-in-margin threshold.

\textbf{5.~Empirical validation} (Section~\ref{sec:numerical}). Numerical experiments on five families of MINLPs confirm each theoretical prediction in turn, and a downlink beamforming instance illustrates the application of the taxonomy in a wireless setting.

\subsection{Implications for Algorithm Design}
\label{sec:discussion:implications}

The most immediate algorithmic implication is that the regime label, as a constant-time-computable structural certificate of the converged solution, can be used to dispatch tailored solvers, surrogates, or step-size rules. 
A regime-aware meta-algorithm operates by interleaving short solver bursts with classifier calls and switching its inner strategy when the regime stabilizes or transitions. 
The conceptual outline is as follows:

\begin{itemize}
\item In $\mathcal{R}_{\mathrm{unc}}$, where the geometry is gradient-dominated, unaccelerated first-order methods (BCD, gradient descent) are competitive, and aggressive step sizes are admissible.
\item In $\mathcal{R}_{\mathrm{res}}$, the few active resource constraints suggest specialized handling: active-set updates, dual decomposition along the binding constraints, or projection onto the active manifold.
\item In $\mathcal{R}_{\mathrm{sat}}$, with many variables at bounds, coordinate methods with active-set acceleration and reduced-space inner solvers are appropriate.
\item In $\mathcal{R}_{\mathrm{coup}}$, with significant inter-block coupling, primal--dual splitting methods (ADMM, PDHG) and consensus-aware preconditioners are warranted.
\item In $\mathcal{R}_{\mathrm{hyb}}$, regime detection signals proximity to a transition and may motivate switching strategies or model refinement, since the iterate is near a structurally ambiguous point identified by Theorem~\ref{thm:hybrid}.
\end{itemize}

A formal development of regime-aware meta-algorithms, including convergence-rate analysis under regime switching, is left to future work.

\subsection{Limitations}
\label{sec:discussion:limitations}

We note four limitations of the present framework, each suggesting a natural avenue for refinement.

\textit{Constraint qualifications.} The results of Section~\ref{sec:structural} are developed under Assumption~\ref{ass:LICQSOSC} (LICQ + SOSC), augmented by strict complementarity where needed (Theorems~\ref{thm:stability}, \ref{thm:transition}). 
These conditions exclude problems with degenerate active sets, in which multipliers may be non-unique or the KKT solution map may fail to be $C^1$. 
Extending the regime taxonomy to the MFCQ setting, where multipliers form a polytope, is an interesting direction; the natural generalization would classify the multiplier polytope rather than a unique multiplier vector.

\textit{Computation of structural constants.} 
The theorems are stated in terms of three constants---the KKT-system conditioning $\kappa$, the feature-map Lipschitz constant $C_\Psi$, and the feature-space margin $\rho_r$---that are problem-dependent and not always available in closed form.
While $\rho_r$ is directly computable from a single solve (Definition~\ref{def:feat-margin}), $\kappa$ and $C_\Psi$ require either structural problem analysis or empirical estimation. 
Refined bounds in specific problem classes (e.g., quadratically constrained quadratic programs, separable problems, semidefinite programs) would tighten the operational forms of Theorems~\ref{thm:classifier-correctness}--\ref{thm:online}.

\textit{Parameter selection.} 
The regime classification depends on the quadruple $\Theta = (\delta,\theta,\gamma,k)$.
We have demonstrated robustness under variations of $\Theta$ in the experiments, but a principled procedure for selecting $\Theta$ from problem data---ideally with theoretical guarantees on the resulting classification---remains open. 
One promising direction is to choose thresholds adaptively via clustering on the empirical multiplier distribution, with consistency guarantees inherited from density-estimation theory.

\textit{Global versus local guarantees.} 
The taxonomy classifies stationary points, which in the non-convex case may not be global minima. 
The regime of a local minimum need not equal the regime of the global optimum, and the classifier output thus inherits the non-convexity limitations of the underlying solver. 
This is unavoidable in the general non-convex setting, but in specific structured non-convex problems (e.g., problems with hidden convexity, or problems amenable to convex relaxation) sharper statements may be possible.

\subsection{Future Directions}
\label{sec:discussion:future}

The framework of this paper opens several lines of inquiry, falling broadly into three categories.

\paragraph*{Theoretical extensions}
The structural theory of regimes admits several natural extensions. First, the MFCQ generalization noted above. 
Second, an extension to non-smooth objectives via subdifferential calculus and the Clarke stationarity condition, which would broaden the applicability to non-smooth machine learning losses. 
Third, a more refined analysis of the Hybrid regime as a \emph{stratified} object: $\mathcal{R}_{\mathrm{hyb}}$ decomposes by which boundary inequality is tight, and each stratum has its own local geometry. 
Fourth, the development of higher-order regime invariants (beyond the multiplier shape) that capture finer structural distinctions, perhaps through second-order Lagrange-multiplier sensitivities.

\paragraph*{Hierarchical Regime Refinement}
Develop a multi-level taxonomy by introducing operationally meaningful sub-regimes (e.g., power-limited, interference-limited, and QoS-limited variants inside the resource-limited regime).
This extension is particularly relevant for wireless resource allocation problems and will be explored in a companion paper.

\paragraph*{Algorithmic extensions}
On the algorithmic side, the natural next step is a regime-aware meta-algorithm with provable rate guarantees. 
A specific question is whether the worst-case rate of such a meta-algorithm can match the best regime-specific rate, and whether the overhead of regime detection is asymptotically dominated by the per-regime solver work. 
Stochastic optimization is a particularly fertile setting: the sample-complexity bound of Theorem~\ref{thm:sample-complexity} is the natural primitive for finite-sample stochastic optimization, and regime-aware variance-reduction methods may inherit improved sample complexity in low-margin instances.

\paragraph*{Application domains}
The taxonomy applies in principle to any constrained nonlinear optimization problem, but several application domains are particularly natural. 
In wireless communications, the natural problems---resource allocation, beamforming, power control, scheduling, network slicing---exhibit precisely the constraint structure (resource budgets, QoS bounds, coupling constraints) that the regimes were designed to identify. The classifier-aided algorithms developed under the framework of this paper provide structural primitives for wireless system design under uncertainty. A series of forthcoming companion papers develops these instantiations in detail. 
Beyond wireless, the framework is applicable to power-systems optimization, optimal control under constraints, structured machine learning (e.g., constrained reinforcement learning), and operations research more broadly.

\paragraph*{Cross-disciplinary connections}
The equilibrium reading of Section~\ref{sec:games}, together with the multiplier-based structural theory, invites connections to neighboring fields. 

In economics, the regime taxonomy parallels classifications of market equilibria by which scarcity constraints bind; in control, the transition theorem resonates with the theory of switching systems and hybrid dynamical systems; in learning theory, the no-regret view of Mirror Descent and Frank--Wolfe (Section~\ref{sec:games:md}, Section~\ref{sec:games:fw}) suggests regret bounds under regime switching, extending recent work on adversarial online optimization. 
We see these cross-disciplinary connections as the most exciting frontier opened by the present work.


\section{Conclusion}
\label{sec:conclusion}

We have presented a structural theory of operational regimes for constrained optimization, motivated by a unified game-theoretic reading of eight classical algorithm families. 
The central object is the regime classification map $\Psi$, which assigns to each KKT multiplier vector one of five labels (Unconstrained, Resource-Limited, Saturation, Strongly-Coupled, Hybrid) based on four scale-free features. 
Four structural theorems establish that this classification is invariant under the natural symmetries of the KKT system, locally constant under perturbation of problem data, characterized by isolated codimension-one transitions along data paths, and bordered by a topologically precise Hybrid set. 
A regime classifier (Algorithm~\ref{alg:classify}) computes $\Psi$ in linear time per call and admits four guarantees---deterministic correctness, classifier-stabilization along solver trajectories, sample complexity, and online tracking under drift---that turn the taxonomy into a computational primitive available to downstream papers.


The unifying message is that, beneath the surface differences among optimization algorithms, the multipliers they compute carry a compact, method-independent structural signature. 
Recognizing this signature opens a path to regime-aware algorithm design, to robust certification of solution structure under noise and drift, and to a structural language that connects the theory of constrained optimization with the equilibrium theory of games.
We see the present paper as a foundation on which a series of follow-up works, both theoretical and applied, will build.

\bibliography{sn-bibliography}

\end{document}